\documentclass[runningheads]{llncs}
\usepackage[T1]{fontenc}
\usepackage{graphicx}
\usepackage{subfigure}
\usepackage{amsmath,amssymb}
\usepackage[algosection,ruled,lined,linesnumbered,longend]{algorithm2e}
\usepackage{hyperref}
\usepackage{cleveref}

\let\proof\relax
\usepackage{amsthm}
\usepackage{tikz, tikz-3dplot}
\usepackage{tikz-qtree}
\usetikzlibrary{positioning,calc}
\usepackage{etoolbox}
\usetikzlibrary{decorations.markings}
\usetikzlibrary{decorations.text}

\newcommand{\rounddown}[1]{\left\lfloor#1\right\rfloor}
\newcommand{\roundup}[1]{\left\lceil#1\right\rceil}

\newcommand{\R}{\mathbb{R}}
\newcommand{\Z}{\mathbb{Z}}
\DeclareMathOperator{\conv}{conv}
\DeclareMathOperator{\cone}{cone}

\begin{document}
%

%
%

\title{Approximately Packing Dijoins via Nowhere-Zero Flows}
%
%
\author{Gérard Cornuéjols \and
Siyue Liu \and
R. Ravi}

%
\authorrunning{Gérard Cornuéjols et al.}
%
\institute{Carnegie Mellon University, Pittsburgh, USA\\
\email{\{gc0v,siyueliu,ravi\}@andrew.cmu.edu}}
\maketitle              

\begin{abstract}
In a digraph, a dicut is a cut where all the arcs cross in one direction. A dijoin is a subset of arcs that intersects each dicut. Woodall conjectured in 1976 that in every digraph, the minimum size of a dicut equals to the maximum number of disjoint dijoins. However, prior to our work, it was not even known whether at least $3$ disjoint dijoins exist in an arbitrary digraph whose minimum dicut size is sufficiently large. By building connections with nowhere-zero (circular) $k$-flows, we prove that every digraph with minimum dicut size $\tau$ contains $\rounddown{\frac{\tau}{k}}$ disjoint dijoins if the underlying undirected graph admits a nowhere-zero (circular) $k$-flow. The existence of nowhere-zero $6$-flows in $2$-edge-connected graphs (Seymour 1981) directly leads to the existence of $\rounddown{\frac{\tau}{6}}$ disjoint dijoins in a digraph with minimum dicut size $\tau$, which can be found in polynomial time as well. The existence of nowhere-zero circular $\frac{2p+1}{p}$-flows in $6p$-edge-connected graphs (Lov\'asz et al. 2013) directly leads to the existence of $\rounddown{\frac{\tau p}{2p+1}}$ disjoint dijoins in a digraph with minimum dicut size $\tau$ whose underlying undirected graph is $6p$-edge-connected.
We also discuss reformulations of Woodall's conjecture into packing strongly connected orientations.
\end{abstract}

\section{Introduction}
Given a digraph $D=(V,A)$ and a subset $U$ of its vertices with $U\neq\emptyset , V$, denote by $\delta_D^+(U)$ and $\delta_D^-(U)$ the arcs leaving and entering $U$, respectively. The cut induced by $U$ is $\delta_D(U):=\delta_D^+(U)\cup \delta_D^-(U)$. We omit the subscript $D$ if the context is clear. 
A \textit{dicut} is an arc subset of the form $\delta^+(U)$ such that $\delta^-(U)=\emptyset$. A \textit{dijoin} is an arc subset $J\subseteq A$ that intersects every dicut at least once. 
More generally, we will also work with the notion of a \textit{$\tau$-dijoin}, which is a subset $J\subseteq A$ that intersects every dicut at least $\tau$ times. If $D$ is a weighted digraph with arc weights $w:A\rightarrow \mathbb{Z}_+$, we say that $D$ with weight $w$ can \textit{pack} $k$ dijoins if there exist $k$ dijoins $J_1,...,J_k$ such that no arc $e$ is contained in more than $w(e)$ of these $k$ dijoins.
In this case, we say that $J_1,...,J_k$ is a \textit{packing} of $D$ under weight $w$. The number $k$ is the \textit{value} of the packing. If every arc $e$ is contained in exactly $w(e)$ of these $k$ dijoins, we say that $D$ with weight $w$ can be 
 \emph{decomposed} into $k$ dijoins. When the digraph is \textit{unweighted}, i.e., $w(e)=1$ for every $e\in A$, $D$ with weight $w$ can pack $k$ dijoins if and only if $D$ contains $k$ (arc-)disjoint dijoins.
 
 More generally, we say that $D$ with weight $w$ can \emph{pack} $k$ digraphs with weight $w^1, \ldots, w^k \in \mathbb{Z}^A_+$ if $w^1 + \ldots + w^k\leq w$. We say that $D$ with weight $w$ can be \emph{decomposed} into $k$ digraphs with weight $w^1, \ldots , w^k$ if  $w^1 + \ldots + w^k=w$.
Edmonds and Giles \cite{edmonds1977min} conjectured the following.
\begin{conjecture}[Edmonds-Giles]\label{conj:E-G}
    Let $D=(V,A)$ be a digraph with arc weights $w\in\{0,1\}^A$. If the minimum weight of a dicut is $\tau$, then $D$ can pack $\tau$ dijoins.
\end{conjecture}
We can assume without loss of generality that $w\in\{0,1\}^A$ because we can always replace an arc $e$ with integer weight $w(e)>1$ by $w(e)$ parallel arcs of weight $1$. Note that the weight $0$ arcs cannot be removed because they, together with the weight $1$ arcs, determine the dicuts.
The above conjecture was disproved by Schrijver \cite{schrijver1980counterexample}. However,
the following unweighted version of the Edmonds-Giles conjecture, proposed by Woodall \cite{woodall1978menger},  is still open.
\begin{conjecture}[Woodall]\label{conj:Woodall}
    In every digraph, the minimum size of a dicut equals the maximum number of disjoint dijoins.
\end{conjecture}

Several weakenings of Woodall's conjecture have been made in the literature. It has been conjectured that there exists some integer $\tau\geq 3$ such that every digraph with minimum dicut size at least $\tau$ contains $3$ disjoint dijoins \cite{openproblemtau3}. Shepherd and Vetta \cite{shepherd2005visualizing} raised the following question. Let $f(\tau)$ be the maximum value such that every weighted digraph whose dicuts all have weight at least $\tau$, can pack $f(\tau)$ dijoins. They conjectured that $f(\tau)$ is of order $\Omega(\tau)$. In this paper, which is the full version of the extended abstract \cite{CLR}, we give an affirmative answer to this conjecture in the \textit{unweighted} case. The main results are the following approximate versions of Woodall's conjecture.
\begin{theorem}\label{thm:approx1}
    Every digraph $D=(V,A)$ with minimum dicut size $\tau$ contains $\rounddown{\frac{\tau}{6}}$ disjoint dijoins, and such dijoins can be found in polynomial time.
\end{theorem} 
Given a digraph $D=(V,A)$, the \textit{underlying undirected graph} is the graph with vertex set $V$ and edge set obtained by replacing each arc $(u,v)\in A$ with an undirected edge $\{u,v\}$. To exclude the cases $\tau=0$ and $\tau=1$, when Woodall's conjecture holds trivially, we assume $\tau\geq 2$ throughout the paper, which implies that the underlying undirected graph is $2$-edge-connected. The following theorem shows that, as the connectivity of the underlying undirected graph increases, one can get a better approximation ratio.
\begin{theorem}\label{thm:approx2}
    Let $p$ be a positive integer. Every digraph $D=(V,A)$ with minimum dicut size $\tau$ and with the property that its underlying undirected graph is $6p$-edge-connected contains $\rounddown{\frac{\tau p}{2p+1}}$ disjoint dijoins.
\end{theorem}
A seminal theorem of Nash-Williams \cite{nash1964decomposition} states that every $2k$-edge-connected graph contains $k$ disjoint spanning trees. Since every spanning tree intersects every cut, in particular it intersects every dicut, and thus every spanning tree is a dijoin. As a consequence, if the underlying undirected graph is $6p$-edge-connected, it contains $3p$ disjoint dijoins. Theorem \ref{thm:approx2} improves upon this if $\tau\geq 6p+3$.


The two theorems stated above are consequences of the following main theorem that we prove in this paper.
\begin{theorem}\label{thm:pack_dijoins}
    For a digraph $D=(V,A)$ with minimum dicut size $\tau$, if the underlying undirected graph admits a nowhere-zero circular $k$-flow, where $k\geq 2$ is a rational number, then $D$ contains $\rounddown{\frac{\tau}{k}}$ disjoint dijoins.
\end{theorem} 

The first ingredient of our approach to proving this result is reducing the problem of packing dijoins in a digraph to that of packing strongly connected digraphs. This reduction is not new and it was already explored by Shepherd and Vetta \cite{shepherd2005visualizing}. 
Augment the input digraph $D$ by adding reverse arcs for all input arcs, and assign weights $\tau$ to the original arcs and $1$ to the newly added reverse arcs. Denote the augmented digraph by $\vec{G}$ with weight $w^D$. Define a \textit{$\tau$-strongly-connected digraph ($\tau$-SCD)} to be a weighted digraph such that the arcs leaving every cut have weight at least $\tau$. A $1$-SCD is a \textit{strongly connected (sub)digraph (SCD)}. It is not hard to see that for a digraph $D$ with minimum dicut size $\tau$, the augmented digraph $\vec{G}$ with weight $w^D$ is $\tau$-strongly-connected. One can then show that packing $\tau' \leq \tau$ dijoins in the original digraph $D$ is equivalent to decomposing the augmented weighted digraph $\vec{G}$ into $\tau'$ strongly connected digraphs (Proposition~\ref{prop:reformulation_SCD}).

We then draw a connection to nowhere-zero flows. Let $G=(V,E)$ be an {\em undirected} graph and let $k \geq 2$ be an integer. Tutte \cite{tutte1954contribution} introduced the notion of a \textit{nowhere-zero $k$-flow}, which is an orientation $E^+$ of its edges, together with an 
assignment of integral flow values between $1$ and $k-1$ to $E^+$ such that the flow is conserved at every node. Goddyn et al. 
\cite{goddyn1998k} extended the definition to allowing $k$ to take fractional values. Let $p,q$ be two integers such that $0<p\leq q$. A \textit{nowhere-zero circular $(1+\frac{q}{p})$-flow} of $G$ is an orientation $E^+$ and an assignment of integral flow values between $p$ and $q$. When $p = 1, q=k-1$ we recover Tutte's notion. We say that an orientation $E^+$ induces a nowhere-zero (circular) $k$-flow if there is such an assignment of flow values to $E^+$. Jaeger \cite{jaeger1976balanced} observed that an orientation induces a nowhere-zero (circular) $k$-flow if and only if $E^+$ is \textit{$k$-cut-balanced}, i.e., $\frac{1}{k}|\delta_{G}(U)|\leq |\delta^+_{E^+}(U)|\leq \frac{k-1}{k}|\delta_{G}(U)|$, $\forall \emptyset \neq U\subsetneqq V$. There is a rich literature on the existence of nowhere-zero (circular) $k$-flows from which we will use two important results: 
Seymour \cite{seymour1981nowhere} showed that there is always a nowhere-zero $6$-flow in $2$-edge-connected graphs, and Younger \cite{younger1983integer} gave a polynomial time algorithm to construct a nowhere-zero $6$-flow in $2$-edge-connected graphs.
Lov\'asz et al. \cite{lovasz2013nowhere} proved that for any positive integer $p$, there always exists a nowhere-zero circular $(2+\frac{1}{p})$-flow in $6p$-edge-connected graphs.

Returning to dijoins and our augmented graph $\vec{G}$, we need to decompose this augmented graph into some $\tau' \leq \tau$ strongly connected digraphs. 
Decomposing a digraph into strongly connected digraphs is a notoriously hard problem. It is not known whether there exists a sufficiently large $\tau$ such that every $\tau$-strongly-connected digraph can be decomposed into $2$ disjoint strongly connected digraphs \cite{bang2004decomposing}. Finding a minimum subdigraph that is strongly connected is NP-hard as well \cite{eswaran1976augmentation}. A natural approach to obtain a small-size strongly connected spanning subdigraph is to find a pair of in and out $r$-arborescences from the same root $r$ and to take their union. However, even finding a pair of disjoint in and out arborescences in a given digraph is hard. For instance, it is still open whether there exists a sufficiently large integer $\tau$ such that an arbitrary $\tau$-strongly-connected digraph can pack one in-arborescence and one out-arborescence \cite{bang2009disjoint}.

To get around this difficulty, we reduce our goal to finding {\em two} disjoint subdigraphs of $\vec{G}$, each of which contains $\tau'$ disjoint in or out $r$-arborescences for some fixed root $r$.
The idea of pairing up in- and out-arborescences was already used successfully by Shepherd and Vetta \cite{shepherd2005visualizing} to find a half-integral packing of dijoins of value $\frac{\tau}{2}$. Here, we crucially argue (in Theorem~\ref{thm:decompose_SCD_weightD}) that if the underlying undirected graph of $D$ admits a nowhere-zero $k$-flow, then the digraph $\vec{G}$ with weight $w^D$ can be decomposed into two disjoint $\rounddown{\frac{\tau}{k}}$-SCD's. (Note that we do not prove this statement for any arbitrary $\tau$-SCD but only the specially weighted $\vec{G}$.) The most natural way is to decompose $\vec{G}$ into two digraphs of weights $\frac{w^D}{2}$, which would both be $\rounddown{\frac{\tau}{2}}$-strongly-connected except that $\frac{w^D}{2}$ may not be integral. Our key idea is to use the orientation $E^+$ that induces a nowhere-zero $k$-flow to round $\frac{w^D}{2}$ into an integral vector: $x_e=\roundup{\frac{w^D_e}{2}}$ if $e\in E^+$ and $x_e=\rounddown{\frac{w^D_e}{2}}$ if $e\notin E^+$. We argue that both $x$ and $(w^D-x)$ are $\rounddown{\frac{\tau}{k}}$-SCD's using the fact that $E^+$ is $k$-cut-balanced. 
Using Edmonds' disjoint arborescences theorem \cite{edmonds1973edge}, we can now extract $\rounddown{\frac{\tau}{k}}$ disjoint in $r$-arborescences from the first and the same number of out $r$-arborescences from the second. Pairing them up gives us the final set of $\rounddown{\frac{\tau}{k}}$ strongly connected digraphs. 
Theorems~\ref{thm:approx1} and \ref{thm:approx2} then follow from the prior theorems about the existence of nowhere-zero flows.

\vspace{1em}
In Section \ref{sec:SCO}, we give equivalent forms of Woodall's conjecture and of the Edmonds-Giles conjecture, respectively, in terms of packing strongly connected orientations, which are of independent interest. Given an undirected graph $G=(V,E)$, let $\vec{G}=(V,E^+\cup E^-)$ be a digraph obtained from making two copies of each edge $e\in E$ and directing them oppositely, one arc being denoted by $e^+\in E^+$ and the other by $e^-\in E^-$. A \textit{$\tau$-strongly connected orientation ($\tau$-SCO)} of $G$ is a multi-subset of arcs from $E^+\cup E^-$ picking exactly $\tau$ many of $e^+$ and $e^-$ (possibly with repetitions) for each $e$ such that at least $\tau$ arcs leave every cut. In particular, a \textit{strongly connected orientation (SCO)} of $G$ is a $1$-SCO of $G$. 
One may ask whether a $\tau$-SCO can always be decomposed into $\tau$ disjoint SCO's. This is not the case. Indeed,
we prove in Theorem \ref{thm:eq_E-G} that
this question is equivalent to the Edmonds-Giles conjecture.
In contrast, we define $x$ to be a \textit{nowhere-zero $\tau$-SCO} if it is a $\tau$-SCO and $x_e\geq 1$ for every arc $e$.
In Theorem \ref{thm:eq_Woodall}, we prove that Woodall's conjecture is true if and only if for every undirected graph $G$, a nowhere-zero $\tau$-SCO can be decomposed into $\tau$ disjoint SCO's.

In Section \ref{sec:Decompose_SCD}, we extend decomposing the special weight $w^D$ into decomposing an arbitrary integral weight $w$ that is a \textit{nowhere-zero $\tau$-SCD}, i.e., an SCD such that $w_e\geq 1$ for every arc $e$. We show that if the underlying undirected graph admits a nowhere-zero (circular) $k$-flow for some rational number $k\geq 2$, then every nowhere-zero $\tau$-SCD can be decomposed into $\rounddown{\frac{\tau}{k+1}}$ strongly connected digraphs.

\subsection*{Related Work}
Shepherd and Vetta \cite{shepherd2005visualizing} raised the question of approximately packing dijoins.
They also introduced the idea of adding reverse arcs to make the digraph $\tau$-strongly-connected, then packing strongly connected subdigraphs, and finally pairing up in- and out-arborescences. 
Yet, this approach itself only gives a half
integral packing of value $\frac{\tau}{2}$ in a digraph with minimum dicut size $\tau$. 
It is conjectured by Kir{\'a}ly \cite{kiraly2007result} that every digraph with minimum dicut size $\tau$ contains two disjoint $\rounddown{\frac{\tau}{2}}$-dijoins, see also \cite{abdi2024arc}. One might notice that if this conjecture is true, together with the approach of combining in and out $r$-arborescences, one can show that there exist $\rounddown{\frac{\tau}{2}}$ disjoint dijoins in a digraph with minimum dicut size $\tau$. 
Abdi et al. \cite{abdi2023packing} proved that every digraph can be decomposed into a dijoin and a $(\tau-1)$-dijoin. Abdi et al. \cite{abdi2024arc} further showed that a digraph with minimum dicut size $\tau$ can be decomposed into a $k$-dijoin and a $(\tau-k)$-dijoin for every integer $k\in\{1,...,\tau-1\}$ under the condition that the underlying undirected graph is $\tau$-edge-connected. M\'esz\'aros \cite{meszaros2015note} proved that when the underlying undirected graph is $(q-1,1)$-partition-connected for some prime power $q$, the digraph can be decomposed into $q$ disjoint dijoins. However, none of these approaches tell us how to decompose a digraph with minimum dicut size $\tau$ into a large number of disjoint dijoins without connectivity requirements. We also refer to the papers that view the problem from the perspective of reorienting the directions of a subset of arcs to make the graph strongly connected, such as \cite{abdi2024arc,chudnovsky2016disjoint,schrijverobervation}. For the context of nowhere-zero $k$-flow, we refer interested readers to \cite{jaeger1976nowhere,jaeger1979flows,jaeger1984circular,lovasz2013nowhere,seymour1981nowhere,thomassen2012weak,younger1983integer} and the excellent survey by Jaeger \cite{jaeger1988nowhere}. 
Finally, Schrijver's unpublished notes \cite{schrijverobervation} reformulate Woodall's conjecture into the problem of partitioning the arcs of the digraph into strengthenings.  A \textit{strengthening} is an arc set $J \subseteq A$ which, on flipping the orientation of the arcs in $J$, makes the digraph strongly connected. This inspired the reformulations in Theorem \ref{thm:eq_E-G} and Theorem \ref{thm:eq_Woodall}.

\section{Technical Background}

\paragraph{Notation.} Given a graph $G=(V,E)$, for a node subset $U\subseteq V$, let $\delta_G(U)$ denote the cut induced by $U$. For an edge subset $F\subseteq E$, let $\delta_F(U):=\delta_G(U)\cap F$ denote the edges in $F$ that are also in the cut induced by $U$. 
Given a digraph $D=(V,A)$, for a node subset $U\subseteq V$, denote by $\delta_D^+(U)$ and $\delta_D^-(U)$ the arcs leaving and entering $U$ in $D$, respectively. Denote by $\delta_D(U):=\delta_D^+(U)\cup \delta_D^-(U)$ the cut induced by $U$. Similarly, for an arc subset $B\subseteq A$, denote by $\delta_B^+(U):=\delta_D^+(U)\cap B$, $\delta_B^-(U):=\delta_D^-(U)\cap B$, and $\delta_B(U):=\delta_D(U)\cap B$. 
Denote by $e^{-1}$ the reverse of arc $e\in A$ and $B^{-1}$ the arcs obtained from reversing the directions of the arcs in $B\subseteq A$. 
Given an undirected graph $G=(V,E)$, make two copies of each edge $e=\{u,v\}\in E$ and direct one to be $e^+=(u,v)$ and the other to be $e^-=(v,u)$ (the assignment of $(u,v)$ and $(v,u)$ to $e^+$ and $e^-$ is arbitrary). Denote by $E^+:=\cup_{e\in E}\{e^+\}$ and $E^-:=\cup_{e\in E}\{e^-\}$. Each of $E^+$ and $E^-$ is an \textit{orientation} of $G$ satisfying $E^-=(E^+)^{-1}$. Denote the resulting digraph by $\vec{G}=(V,E^+\cup E^-)$.
For a function $f:E\rightarrow \R$, denote by $f(F):=\sum_{e\in F}f(e)$, $\forall F\subseteq E$. For a subset $F\subseteq E$, denote by $\chi_F\in\{0,1\}^E$ the characteristic vector of $F$.

\subsection{Strongly Connected Digraphs and Arborescences}\label{sec:SCD}

Let $G=(V,E)$ be an undirected graph and let $\vec{G}=(V,E^+\cup E^-)$ be the digraph obtained by copying each edge of $G$ twice and directing them oppositely. Recall that strongly connected (sub)digraph (SCD) of $G$ is a subset of arcs $F\subseteq E^+\cup E^-$ that spans $V$ and the digraph induced by $F$ is strongly connected. Equivalently, $|\delta^+_{F}(U)|\geq 1$, for every $\emptyset \neq U\subsetneqq V$. For an integral vector $x\in\Z^{E^+\cup E^-}_+$, we say $x$ is a $\tau$-strongly connected digraph ($\tau$-SCD) of $G$ if it satisfies $|x(\delta^+_{\vec{G}}(U))|\geq \tau, \forall \emptyset \neq U\subsetneqq V$.

We recall a classical theorem about decomposing digraphs into arborescences.
Let $D=(V,A)$ be a digraph and let $r\in V$ be a \textit{root}. An \textit{out (resp. in) $r$-arborescence} is 
a digraph such that for every $v\in V\setminus \{r\}$, there is exactly one directed walk from $r$ to $v$ (resp. $v$ to $r$). A digraph $D$ is called an \textit{out (in) arborescence} if there exists a vertex $r$ such that $D$ is an out (in) $r$-arborescence. Edmonds's disjoint arborescences theorem \cite{edmonds1973edge} states that when fixing a root $r$, every \textit{rooted-$\tau$-connected} digraph, i.e., $|\delta_D^+(U)|\geq \tau, \forall U\subsetneqq V, r\in U$, contains $\tau$ disjoint out $r$-arborescences. Furthermore, this decomposition can be computed in polynomial time.
\begin{theorem}[\cite{edmonds1973edge}]\label{thm:edmonds}
    Given a digraph $D$ and a root $r$, if $D$ is rooted-$\tau$-connected, then $D$ contains $\tau$ disjoint out $r$-arborescences, and such $r$-arborescences can be found in polynomial time.
\end{theorem}

A $\tau$-strongly-connected digraph is in particular rooted-$\tau$-connected. Therefore, for every fixed root $r\in V$, a $\tau$-strongly-connected digraph contains $\tau$ disjoint out $r$-arborescences. If we reverse the directions of the arcs and apply Theorem \ref{thm:edmonds}, we see that a $\tau$-strongly-connected digraph also contains $\tau$ disjoint in $r$-arborescences.

\begin{corollary}\label{cor:edmonds}
     Given a $\tau$-strongly-connected digraph $D$ and a root $r$, $D$ contains $\tau$ disjoint out (in) $r$-arborescences,  and such $r$-arborescences can be found in polynomial time.
\end{corollary}


\subsection{Nowhere-Zero Flows and Cut-Balanced Orientations}
Let $G=(V,E)$ be an undirected graph and $k$ an integer such that $k\geq 2$. Recall that a nowhere-zero $k$-flow of $G$ is an orientation $E^+$ and a function $f:E^+ \rightarrow \{1,2,...,k-1\}$ such that $f(\delta_{E^+}^+(v))=f(\delta_{E^+}^-(v))$ for every vertex $v \in V$. 
Many efforts have been made to find the smallest $k$ such that there exists a nowhere-zero $k$-flow (see e.g. in \cite{jaeger1988nowhere}). Tutte made the following $5$-flow conjecture.
\begin{conjecture}[Tutte, 5-flow]\label{conj:5-flow}
    Every $2$-edge-connected graph has a nowhere-zero $5$-flow.
\end{conjecture}








Relaxations of Conjectures \ref{conj:5-flow} have been proved.
Jaeger \cite{jaeger1976nowhere,jaeger1979flows} proved that every $2$-edge-connected graph admits a nowhere-zero $8$-flow, which was improved by Seymour \cite{seymour1981nowhere} showing that there is always a nowhere-zero $6$-flow in $2$-edge-connected graphs. Younger \cite{younger1983integer} gave an algorithmic version of this result.

\begin{theorem}[\cite{seymour1981nowhere,younger1983integer}]\label{thm:6-flow}
    Every $2$-edge-connected graph admits a nowhere-zero $6$-flow which can be found in polynomial time.
\end{theorem}

Let $p,q$ be two integers such that $0<p\leq q$. A \textit{nowhere-zero circular $(1+\frac{q}{p}$)-flow} of $G$ is an orientation $E^+$ and $f:E^+\rightarrow \{p,p+1,...,q\}$, such that $f(\delta_{E^+}^+(v))=f(\delta_{E^+}^-(v))$. 
In fact, $f$ is allowed to take any real values in $[p,q]$ (or in $[1,\frac{q}{p}]$ after rescaling). This is because the flow polytope $\big\{f\in \R^{E^+}:f(\delta_{E^+}^+(v))=f(\delta_{E^+}^-(v)),\ p\leq f\leq q\big\}$ is integral, and thus the existence of a fractional $f$ is equivalent to the existence of an integral $f$.
It was shown that if $G$ admits a nowhere-zero circular $\alpha$-flow, then it also admits a nowhere-zero circular $\beta$-flow for every $\beta\geq \alpha$ \cite{goddyn1998k}.

It turns out that if the connectivity is sufficiently high, the graph admits a nowhere-zero circular $k$-flow for some $k>2$ that is arbitrarily close to $2$. Thomassen \cite{thomassen2012weak} proved that there is a function $Q:\R_+\rightarrow \R_+$ such that every $Q(p)$-edge-connected graph admits a nowhere-zero circular $(2+\frac{1}{p})$-flow, where $Q(p)=\Theta(p^2)$. 
Lov\'asz et al. \cite{lovasz2013nowhere} strengthened the result by reducing the connectivity requirement to linear as follows.

\begin{theorem}[\cite{lovasz2013nowhere}]\label{thm:6k-connected:circular flow}
    Every $6p$-edge-connected graph admits a nowhere-zero circular $(2+\frac{1}{p})$-flow.
\end{theorem}


Nowhere-zero integer flow and circular flow are closely related to cut-balanced orientations. Recall that, for $k \geq 2$, an orientation $E^+$ is $k$-cut-balanced if $\frac{1}{k}|\delta_{G}(U)|\leq |\delta^+_{E^+}(U)|\leq \frac{k-1}{k}|\delta_{G}(U)|$, $\forall \emptyset \neq U\subsetneqq V$. This is equivalent to saying that both $\delta_{E^+}^+(U)$ and $\delta_{E^+}^-(U)$ 
contain at least a fraction $\frac{1}{k}$ of the edges 
in $\delta_{G}(U)$. It is clear that $E^+$ is $k$-cut-balanced if and only if $E^-=(E^+)^{-1}$ is $k$-cut-balanced. Jaeger \cite{jaeger1976balanced} proved that $E^+$ is a $k$-cut-balanced orientation if and only if it induces a nowhere-zero $k$-flow. This observation has also been pointed out in different places (e.g. see \cite{goddyn2001open},\cite{goddyn1998k}, \cite{thomassen2012weak})). Since we repeatedly use it, we provide its proof.


\begin{lemma}[\cite{jaeger1976balanced}]\label{lem:key}
    Given an undirected graph $G = (V,E)$ and a rational $k\geq 2$, an orientation induces a nowhere-zero circular $k$-flow if and only if it is $k$-cut-balanced.
\end{lemma}

\proof
    ``Only if" direction: Let $E^+$ and $f:E^+\rightarrow [1,k-1]$ be a nowhere-zero circular $k$-flow. By flow conservation, $f(\delta^+_{E^+}(U))=f(\delta^-_{E^+}(U)),\forall U \subsetneqq V, U\neq\emptyset$. Thus, one has $1\cdot |\delta^+_{E^+}(U)|\leq f(\delta^+_{E^+}(U))=f(\delta^-_{E^+}(U))\leq (k-1)\cdot |\delta^-_{E^+}(U)|$. Similarly, one also has $1\cdot |\delta^-_{E^+}(U)|\leq f(\delta^-_{E^+}(U))=f(\delta^+_{E^+}(U))\leq (k-1)\cdot |\delta^+_{E^+}(U)|$. It follows from the equality $|\delta_G(U)|=|\delta_{E^+}^+(U)|+|\delta_{E^+}^-(U)|$ that $\frac{1}{k}|\delta_{G}(U)|\leq |\delta^+_{E^+}(U)|\leq \frac{k-1}{k}|\delta_{G}(U)|$. Thus, $E^+$ is a $k$-cut-balanced orientation.

    ``If" direction: Let $E^+$ be a $k$-cut-balanced orientation. By Hoffman's circulation theorem \cite{hoffman2003some}, there exists a circulation $f\in\R^{E^+}$ satisfying $1\leq f_e\leq k-1$, $\forall e\in E^+$, if and only if $|\delta_{E^+}^-(U)|\leq (k-1)|\delta_{E^+}^+(U)|$, $\forall \emptyset \neq U\subsetneqq V$, which is equivalent to the $k$-cut-balanced condition $|\delta_{E^+}^+(U)|\geq \frac{1}{k}|\delta_{G}(U)|$, $\forall \emptyset \neq U\subsetneqq V$. Thus, $E^+$ together with $f$ is a nowhere-zero circular $k$-flow.
\qedsymbol

\subsection{Submodular Flows}
A family $\mathcal{C}$ of subsets of a set $V$ is a \emph{crossing family} if for all $U, W\in \mathcal{C}$ such that $U\cap W\neq \emptyset$ and $U\cup W\neq V$, one has $U\cap W, U\cup W\in \mathcal{C}$. A function $f:\mathcal{C}\rightarrow \R$ is \emph{crossing submodular} over a crossing family $\mathcal{C}$ if for every $U,W\in\mathcal{C}$ such that $U\cap W\neq \emptyset$ and $U\cup W\neq V$, one has $f(U\cap W)+f(U\cup W)\leq f(U)+f(W)$. Given a digraph $D=(V,A)$ and an integral crossing submodular function $f:\mathcal{C}\rightarrow \Z$ over a crossing family $\mathcal{C}$, a \emph{submodular flow} is a vector $x$ in the following polyhedron. 
\[
P(f):=\big\{x\in \R^A \mid x(\delta^-(U))-x(\delta^+(U))\leq f(U),\ \forall U\in\mathcal{C}\big\}.
\]
A polyhedron $P$ is \emph{box-integral} if $P\cap \{x: l\leq x\leq u\}$ is integral for every $l,u\in \Z$. 
The seminal work of Edmonds and Giles \cite{edmonds1977min} shows the following (see also \cite{schrijver2003combinatorial} Chapter 60)).
\begin{theorem}[\cite{edmonds1977min}]\label{thm:submodular-flow}
     For every integral crossing submodular function $f$, the submodular flow polyhedron $P(f)$ is box-integral.
\end{theorem}


\section{An Approximate Packing of Dijoins}\label{sec:approx}
In this section, we prove our main Theorem~\ref{thm:pack_dijoins}.
Let $D=(V,A)$ be a digraph with minimum dicut size $\tau$.
Let  $k\geq 2$ be a rational.
Assume that the underlying undirected graph of $D$ admits a nowhere-zero circular $k$-flow. 
According to Lemma \ref{lem:key}, the existence of a nowhere-zero (circular) $k$-flow implies that there is a $k$-cut-balanced orientation $E^+$: for each cut, the number of arcs entering it differs from the number of arcs leaving it by a factor of at most $(k-1)$. This implies that the two subdigraphs of $D$ consisting of the arcs that are in the same orientation as $E^+$ and its complement both intersect every dicut in a large proportion of its size. This gives us a way to decompose the digraph into two subdigraphs each intersecting every dicut in a large number of arcs (For an example, see Figure \ref{fig:decompose}). Recall that a
$\tau$-dijoin is an arc subset that intersects each dicut at least $\tau$ times.

\begin{proposition}\label{prop:decompose_1/k_dijoin}
    For a digraph $D=(V,A)$ with minimum dicut size $\tau$, if the underlying undirected graph admits a nowhere-zero circular $k$-flow for some rational number $k\geq 2$, then $D$ contains two disjoint $\rounddown{\frac{\tau}{k}}$-dijoins.
\end{proposition}
\begin{proof}
    Let $E^+$ and $f:E^+\rightarrow[1,k-1]$ be a nowhere-zero circular $k$-flow of the underlying undirected graph $G$ of $D$.
    By Lemma~\ref{lem:key}, $\frac{1}{k}|\delta_{G}(U)|\leq |\delta^+_{E^+}(U)|\leq \frac{k-1}{k}|\delta_{G}(U)|$ for every $U \subsetneqq V, U\neq\emptyset$.
    Take $J=A\cap E^+$ to be the arcs that have the same directions in $A$ and $E^+$. Then, for a dicut $\delta^+_D(U)$ such that $\delta_D^-(U)=\emptyset$, we have $|J\cap \delta_D^+(U)|=|\delta_{E^+}^+(U)|\geq \frac{1}{k}|\delta_{G}(U)|= \frac{1}{k}|\delta_{D}^+(U)|\geq \frac{\tau}{k}$ and $|(A\setminus J)\cap \delta_D^+(U)|=|\delta_D^+(U)|-|\delta_{E^+}^+(U)|\geq  |\delta_D^+(U)|-\frac{k-1}{k}|\delta_{G}(U)| = |\delta_D^+(U)|-\frac{k-1}{k}|\delta_{D}^+(U)|\geq \frac{\tau}{k}$. Thus, both $J$ and $A\setminus J$ are $\rounddown{\frac{\tau}{k}}$-dijoins.
\end{proof}

    \begin{figure}[htbp]
\centering
     \includegraphics[scale=0.13]{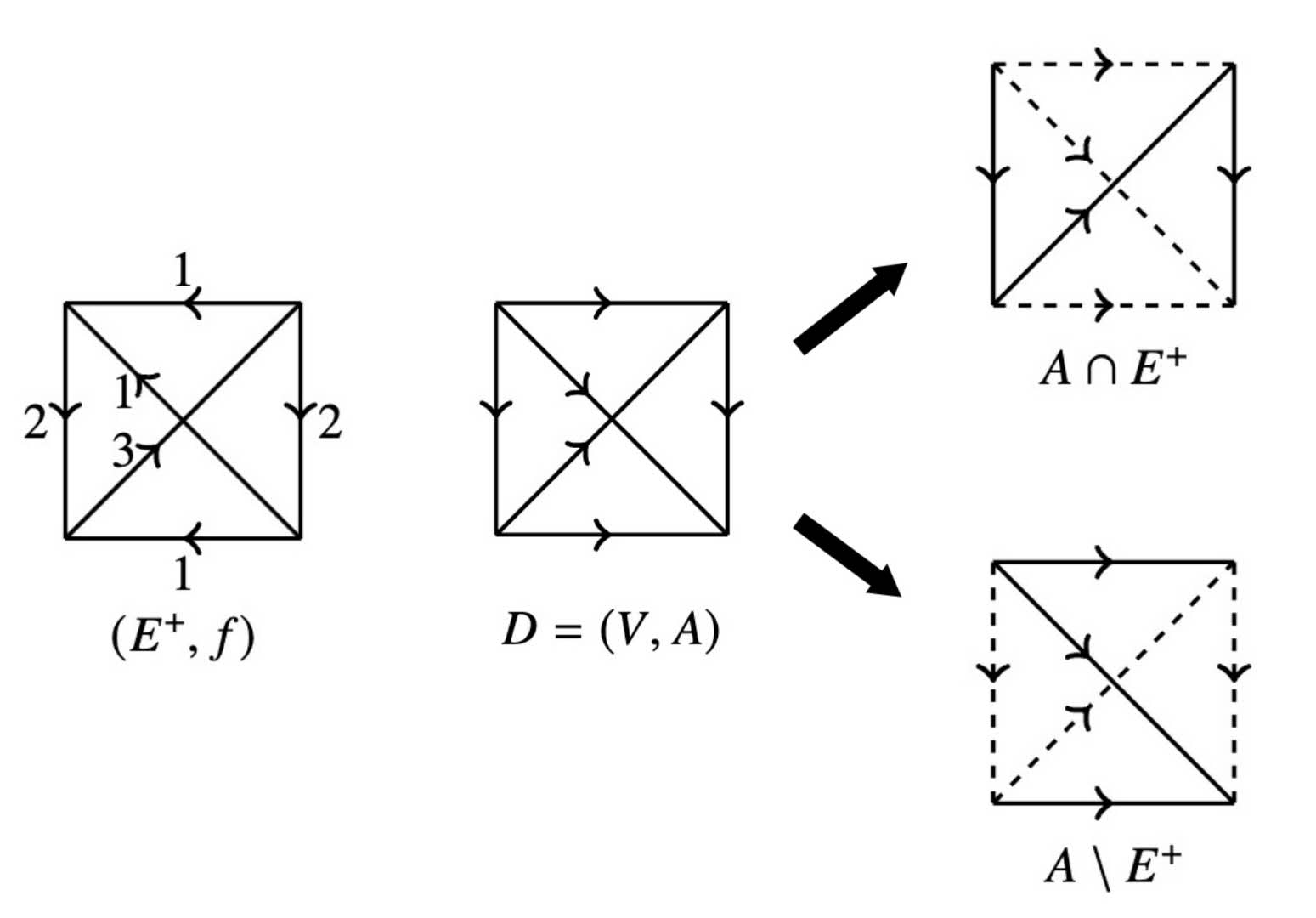}
 \vspace{-1em}
 \caption{$(E^+,f)$ is a nowhere-zero $4$-flow of $G=K_4$. $D=(V,A)$, whose underlying undirected graph is $G$, can be decomposed into a dijoin $A\cap E^+$ and a $2$-dijoin $A\setminus E^+$.}
 \label{fig:decompose}
    \end{figure}

\bigskip
In a digraph $D$ with minimum dicut size $\tau$, although Proposition \ref{prop:decompose_1/k_dijoin} suggests that $D$ can be decomposed into two digraphs, each being a $\rounddown{\frac{\tau}{k}}$-dijoin, there is no guarantee that the new digraphs have minimum dicut size at least $\rounddown{\frac{\tau}{k}}$. This is because a non-dicut in $D$ may become a dicut when we delete arcs, which can potentially have very small size. This is a general difficulty with inductive proofs for decomposing a digraph into dijoins. 

Our key observation here is that, by switching to the setting of strongly connected digraphs, we can bypass this issue.
Given a digraph $D=(V,A)$ with minimum dicut size $\tau$, let $G$ be the underlying undirected graph of $D$ and $\vec{G}=(V,E^+\cup E^-)$ be the digraph obtained by copying each edge of $G$ twice and directing them oppositely. For convenience, we let $E^+=A$ and $E^-=A^{-1}$. 
Define the \textit{weights associated with $D$} to be $w^D\in\Z^{E^+\cup E^-}$ such that $w^D_{e^+}=\tau,\forall e^+\in E^+$ and $w^D_{e^-}=1,\forall e^-\in E^-$. 
It is easy to see that $\vec{G}$ with weight $w^D$ is $\tau$-SCD. Indeed, for every $\emptyset \neq U\subsetneqq V$ such that $\delta_{D}^+(U)\neq\emptyset$, there exists some arc $e^+\in E^+$ such that $e^+\in \delta_{\vec{G}}^+(U)$, and thus $w^D(\delta_{\vec{G}}^+(U))\geq w^D_{e^+}= \tau$. Otherwise, $\delta_{D}^+(U)=\emptyset$ which means $\delta_{D}^-(U)$ is a dicut.
Therefore, $w^D(\delta_{\vec{G}}^+(U))=w^D(\delta_{E^-}^+(U))=|\delta_{D}^-(U)|\geq \tau$. This means that the augmented digraph $\vec{G}$ with weight $w^D$ is $\tau$-strongly-connected. We first reformulate the problem of packing dijoins in $D$ into a problem of packing strongly connected digraphs in $\vec{G}$ under weight $w^D$. We  then prove a decomposition result into two $\lfloor \tau/k \rfloor$-strongly-connected digraphs with the help of nowhere-zero circular $k$-flows. The following reformulation has essentially been stated and used in \cite{shepherd2005visualizing}. We include its proof here.

\begin{proposition}\label{prop:reformulation_SCD}
     For an integer $k\leq \tau$, the digraph $D$ contains $k$ disjoint dijoins if and only if $\vec{G}$ with weight $w^D$ can pack $k$ strongly connected digraphs.
\end{proposition}
\begin{proof}
        Let $F_1,...,F_k$ be $k$ strongly connected digraphs of $G$ that is a packing of $\vec{G}$ under weight $w^D$. Define $J_i:=\{e^+\in E^+\mid \chi_{F_i}(e^-)=1, e^-\in E^-\}$ for every $i\in [k]$. We claim that $J_i$ is a dijoin of $D$ for every $i\in [k]$. Suppose not. Then there exists $i\in [k]$ and a dicut $\delta_D^-(U)$ such that $J_i\cap \delta_D^-(U)=\emptyset$. This implies $F_i\cap \delta_{\vec{G}}^+(U)=\emptyset$, contradicting the fact that $F_i$ is a strongly connected digraph of $\vec{G}$. Moreover, since $w^D_{e^-}=1$, at most one of $F_1,...,F_k$ uses $e^-$, $\forall e^-\in E^-$. Thus, at most one of $J_1,...,J_k$ uses $e^+, \forall e^+\in E^+=A$. Therefore, $J_1,...,J_k$ are disjoint dijoins of $D$. 
    
    Conversely, let $J_1,...,J_k$ be $k$ disjoint dijoins in $D$.
    Let $F_i:=A\cup J_i^{-1}$ for every $i\in [k]$. Observe that each $F_i$ is a strongly connected digraph. This is because every dicut of $A$ has nonempty intersection with the dijoin $J_i$ and thus is no longer a dicut after adding $J_i^{-1}$. It follows from $J_1,...,J_k$ being disjoint and the definition of $w^D$ that $F_1,...,F_k$ is a valid packing in $\vec{G}$ with weight $w^D$.
\end{proof}


\begin{theorem}\label{thm:decompose_SCD_weightD}
     Let $D=(V,A)$ be a digraph with minimum dicut size $\tau$. If the underlying undirected graph $G$ admits a nowhere-zero circular $k$-flow for some rational number $k\geq 2$, then the digraph $\vec{G}$ with weight $w^D$ can pack two $\rounddown{\frac{\tau}{k}}$-SCD's.
\end{theorem}
\begin{proof}
    Let $G=(V,E)$ be the underlying undirected graph of $D$. Let $E^+$ and $f:E^+\rightarrow [1,k-1]$ be a nowhere-zero circular $k$-flow of $G$. Let $E^-$ be obtained by reversing the arcs of $E^+$. Let $\vec{G}=(V,E^+\cup E^-)$. Then, both $E^+$ and $E^-$ are $k$-cut-balanced orientations of $G$ by Lemma \ref{lem:key}. 
    Construct $x\in \Z^{E^+\cup E^-}$ as follows.
        \begin{equation*}
            \begin{aligned}
                x_e=\begin{cases}
                    \roundup{\frac{\tau}{2}}, & e\in A\cap E^+\\
                    \rounddown{\frac{\tau}{2}}, & e\in A\cap E^-\\
                    1, & e\in A^{-1}\cap E^+\\
                    0, & e\in A^{-1}\cap E^-\\
                \end{cases},
            \end{aligned}
            \text{ and equivalently }
            \begin{aligned}
                (w^D-x)_e=\begin{cases}
                    \rounddown{\frac{\tau}{2}}, & e\in A\cap E^+\\
                    \roundup{\frac{\tau}{2}}, & e\in A\cap E^-\\
                    0, & e\in A^{-1}\cap E^+\\
                    1, & e\in A^{-1}\cap E^-\\
                \end{cases}.
            \end{aligned}
        \end{equation*}
    We prove that both $x$ and $(w^D-x)$ are $\rounddown{\frac{\tau}{k}}$-SCD's. Let $\emptyset \neq U \subsetneqq V$. We discuss two cases.
    
    Suppose $|\delta_G(U)|\geq \tau$. Then, $x(\delta_{\vec{G}}^+(U))\geq x(\delta_{E^+}^+(U)) \geq |\delta_{E^+}^+(U)|\geq \frac{1}{k}|\delta_{G}(U)|\geq \frac{\tau}{k}$, where the second inequality follows from the fact that $x_e\geq 1,~\forall e\in E^+$, and the third inequality follows from the fact that $E^+$ is $k$-cut-balanced.
    On the other hand, $(w^D-x)_e\geq 1,~\forall e\in\delta_{E^-}^+(U)$. Therefore, $(w^D-x)(\delta_{\vec{G}}^+(U))\geq (w^D-x)(\delta_{E^-}^+(U))\geq  |\delta_{E^-}^+(U)|\geq \frac{1}{k}|\delta_{G}(U)|\geq \frac{\tau}{k}$, where the third inequality follows from the fact that $E^-$ is $k$-cut-balanced. 
        
    Suppose $|\delta_G(U)|< \tau$. Then, $\delta^-_D(U)$ is not a dicut, i.e., $\delta_D^+(U)\neq \emptyset$, which means  $\delta_{\vec{G}}^+(U)\cap A\neq \emptyset$. Therefore, $x(\delta_{\vec{G}}^+(U))\geq x(\delta_{\vec{G}}^+(U)\cap A)\geq \rounddown{\frac{\tau}{2}}\geq \rounddown{\frac{\tau}{k}}$ since $k\geq 2$. Also, $(w^D-x)(\delta_{\vec{G}}^+(U))\geq (w^D-x)(\delta_{\vec{G}}^+(U)\cap A)\geq \rounddown{\frac{\tau}{2}}\geq \rounddown{\frac{\tau}{k}}$ since $k\geq 2$. Therefore, both $x$ and $(w^D-x)$ are $\rounddown{\frac{\tau}{k}}$-SCD's.
\end{proof}

\subsection*{Proof of Theorem~\ref{thm:pack_dijoins}}
From Proposition~\ref{prop:reformulation_SCD}, given a digraph $D$, we can reduce packing dijoins of $D$ into packing strongly connected digraphs of the augmented digraph $\vec{G}$ with weight $w^D$ which is $\tau$-strongly-connected.
Note that using Corollary~\ref{cor:edmonds}, this digraph can pack $\tau$ in $r$-arborescences, or $\tau$ out $r$-arborescences. If we pair each in $r$-arborescence with an out $r$-arborescence, we will obtain $\tau$ strongly connected digraphs. However, each arc can be used in both in and out $r$-arborescences. Shepherd and Vetta \cite{shepherd2005visualizing} use this idea to obtain a half integral packing of dijoins of value $\frac{\tau}{2}$. Yet, finding disjoint in and out arborescences together is quite challenging. As we noted earlier, it is open whether there exists a $\tau$ such that every $\tau$-strongly-connected digraph can pack one in-arborescence and one out-arborescence \cite{bang2009disjoint}.  

Theorem \ref{thm:decompose_SCD_weightD} paves a way to approximately packing disjoint in and out arborescences. Fixing a root $r$, if we are able to decompose the graph into two $\tau'$-strongly-connected graphs and thereby find $\tau'$ disjoint in $r$-arborescences in the first graph and $\tau'$ disjoint out $r$-arborescences in the second graph, then we can combine them to get $\tau'$ strongly connected digraphs. 

\begin{proof}[Proof of Theorem \ref{thm:pack_dijoins}]
    By Proposition \ref{prop:reformulation_SCD}, it suffices to prove that $\vec{G}$ with weight $w^D$ can pack $\rounddown{\frac{\tau}{k}}$ strongly connected digraphs. By Theorem \ref{thm:decompose_SCD_weightD}, $\vec{G}$ with weight $w^D$ can pack two weighted digraphs, $J_1$ and $J_2$, such that each of them is $\rounddown{\frac{\tau}{k}}$-strongly-connected. Fixing an arbitrary root $r$, since a $\rounddown{\frac{\tau}{k}}$-strongly-connected digraph is in particular rooted-$\rounddown{\frac{\tau}{k}}$-connected, by Theorem~\ref{thm:edmonds}, $J_1$ can pack $\rounddown{\frac{\tau}{k}}$ out $r$-arborescences $S_1,...,S_{\rounddown{\frac{\tau}{k}}}$. 
    Similarly, $J_2$ can pack $\rounddown{\frac{\tau}{k}}$ in $r$-arborescences $T_1,...,T_{\rounddown{\frac{\tau}{k}}}$. Let $F_i:=S_i\cup T_i$, for $i=1,...,\rounddown{\frac{\tau}{k}}$. Each $F_i$ is a strongly connected digraph. This is because every out $r$-cut $\delta_{\vec{G}}^+(U), r\in U$ is covered by $S_i$ and every in $r$-cut $\delta_{\vec{G}}^+(U), r\notin U$ is covered by $T_i$ and thus every cut $\delta_{\vec{G}}^+(U)$ is covered by $F_i$. Therefore, $F_1,...,F_{\rounddown{\frac{\tau}{k}}}$ forms a packing of strongly connected digraphs under weight $w^D$.
\end{proof}

Theorem~\ref{thm:approx1} now follows by combining Theorem~\ref{thm:pack_dijoins} and Theorem \ref{thm:6-flow} and noting that the underlying undirected graph of a digraph with minimum dicut size $\tau\geq 2$ is $2$-edge-connected. By Theorem \ref{thm:6-flow}, the nowhere-zero $6$-flow can be found in polynomial time, and thus the decomposition described in Theorem \ref{thm:decompose_SCD_weightD} can be done in polynomial time. Moreover, further decomposing $J_1$ and $J_2$ into in and out $r$-arborescences can also be done in polynomial time due to Theorem \ref{thm:edmonds}. Thus, we can find $\rounddown{\frac{\tau}{6}}$ disjoint dijoins in polynomial time. 
Theorem~\ref{thm:approx2} now follows by combining Theorem~\ref{thm:pack_dijoins} and Theorem \ref{thm:6k-connected:circular flow}. However, as far as we know there is no constructive version of Theorem \ref{thm:6k-connected:circular flow}, which means Theorem \ref{thm:approx2} cannot be made algorithmic directly.



\section{A Reformulation of Woodall's Conjecture in terms of Strongly Connected Orientations}\label{sec:SCO}
In this section, we discuss the relation between packing dijoins, strongly connected orientations and strongly connected digraphs. We also discuss another reformulation of Woodall's Conjecture in terms of strongly connected orientations.

    Given an undirected graph $G=(V,E)$, let $\vec{G}=(V,E^+\cup E^-)$ be the digraph obtained by copying each edge of $G$ twice and orienting them oppositely. 
    An orientation of $G$ is a subset of $E^+\cup E^-$ consisting of exactly one of $e^+$ and $e^-$ for each $e\in E$. Note that $E^+$ itself is an orientation, so is $E^-$. An orientation $O$ is a strongly connected orientation (SCO) if $|\delta_O^+(U)|\geq 1$ for every $\emptyset \neq U\subsetneqq V$. Note that each strongly connected orientation is a strongly connected digraph.
Given a directed graph $D=(V,A)$, recall that a strengthening is a subset $J\subseteq A$ such that $(V,(A\setminus J)\cup J^{-1})$ is strongly connected. Note that each strengthening is a dijoin. 
Schrijver observed the following reformulation of Woodall's conjecture in terms of strengthenings in his unpublished note (\cite{schrijverobervation} Section 2).
\begin{theorem}[\cite{schrijverobervation}]\label{thm:eq_strenghthening}
    Woodall's conjecture is true if and only if in every digraph with minimum dicut size $\tau$, the arcs can be partitioned into $\tau$ strengthenings.
\end{theorem}

Fix an orientation $E^+$ and let $G^+=(V,E^+)$. There is a one-to-one mapping between SCO's of $G$ and strengthenings of $G^+$. Given a strengthening $J\subseteq E^+$, $(E^+\setminus J)\cup J^{-1}$ is a strongly connected orientation of $G$. Conversely, given a strongly connected orientation $O\subseteq E^+\cup E^-$, $E^+\setminus O$ is a strengthening of $G^+$.
Let $\chi_O\in \{0,1\}^{E^+\cup E^-}$ be the characteristic vector of an orientation $O$ of $G$ and let
     \[
        SCO(G):=\big\{x\in\{0,1\}^{E^+\cup E^-}\big|x=\chi_O \text{ for some strongly connected orientation $O$ of $G$}\big\}.
    \]
    Let $\chi_J\in \{0,1\}^{E^+}$ be the characteristic vector of a strengthening $J$ of $G^+$ and let
    \[
        STR(G^+):=\big\{x\in\{0,1\}^{E^+}\big|x=\chi_J \text{ for some strengthening $J$ of $G^+$}\big\}.
\]
Let $x\in\{0,1\}^{E^+}$. Observe that $x$ is the characteristic vector of a strengthening if and only if
\[\sum_{e\in \delta_{E^+}^-(U)} (1-x_e)+ \sum_{e\in \delta_{E^+}^+(U)} x_e \geq 1,\ \forall \emptyset \neq U\subsetneqq V,
\]
i.e.,
\[
x(\delta_{E^+}^-(U))-x(\delta_{E^+}^+(U))\leq |\delta_{E^+}^-(U)|-1, \forall \emptyset \neq U\subsetneqq V.
\]
One can easily verify that $|\delta_{E^+}^-(U)|-1$ is a crossing submodular function over the crossing family $2^V\setminus \{\emptyset, V\}$. Therefore, it follows from \Cref{thm:submodular-flow} that the convex hull of the strengthenings equals the following box-constrained submodular flow polytope (see also \cite{abdi2024arc}).
\begin{equation*}
    \begin{aligned}
        \conv(STR(G^+))=\big\{x\in\mathbb{R}^{E^+}\big|~&0\leq x_{e^+}\leq 1,~\forall e\in E,\\
        &x(\delta_{E^+}^-(U))-x(\delta_{E^+}^+(U))\leq |\delta_{E^+}^-(U)|-1, \forall \emptyset \neq U\subsetneqq V\big\}.
    \end{aligned}
\end{equation*}
As we hinted earlier, there is a linear transformation $L:STR(G^+)\rightarrow SCO(G), x\mapsto (x,\mathbf{1}-x)$. Therefore, $\conv(SCO(G))=\conv(L\cdot STR(G^+))=L\cdot\conv(STR(G^+))$, which equals the following.
    \begin{equation*}
    \begin{aligned}
        \conv(SCO(G))=\big\{x\in\mathbb{R}^{E^+\cup E^-}\big|~&x_{e^+}\geq 0,~ x_{e^-}\geq 0, ~\forall e\in E,\\
        &x_{e^+}+x_{e^-}=1,~\forall e\in E,\\
        &x(\delta^+(U))\geq 1, \forall \emptyset \neq U\subsetneqq V\big\}.
    \end{aligned}
    \end{equation*}

    As a generalization, a $\tau$-strongly-connected orientation ($\tau$-SCO) of $G$, where $\tau$ is a positive integer, is an integral vector in the following polytope
    \begin{equation}\label{eq:def_P0^tau}
    \begin{aligned}
        P_0^\tau(G):=\big\{x\in\mathbb{R}^{E^+\cup E^-}\big|~&x_{e^+}\geq 0,~ x_{e^-}\geq 0, ~\forall e\in E,\\
        &x_{e^+}+x_{e^-}=\tau,~\forall e\in E,\\
        &x(\delta^+(U))\geq \tau,~ \forall \emptyset \neq U\subsetneqq V\big\}.
    \end{aligned}
    \end{equation}
    Note that $P_0^1(G)=\conv(SCO(G))$, and $P_0^\tau(G)=\tau\conv(SCO(G))$ is just a scaling of $\conv(SCO(G))$, which is integral if $\tau$ is integral. Thus, $P_0^\tau(G)$ describes the convex hull of $\tau$-SCO's. Taking the union of $P_0^\tau(G)$ for all nonnegative $\tau$'s, we obtain
    \begin{equation}
    \begin{aligned}
        P_0(G):=\big\{x\in\mathbb{R}^{E^+\cup E^-}\big|~\exists \tau, ~s.t. ~&x_{e^+}\geq 0,~ x_{e^-}\geq 0, ~\forall e\in E,\\
        &x_{e^+}+x_{e^-}=\tau,~\forall e\in E,\\
        &x(\delta^+(U))\geq \tau,~ \forall \emptyset \neq U\subsetneqq V\big\}.
    \end{aligned}
    \end{equation}
Note that $P_0(G)=\cone(SCO(G))$ is the cone generated by all the strongly connected orientations of $G$. It is reasonable to ask if $P_0^1(G)$ has the integer decomposition property~\cite{baum1978integer}, i.e. every integral vector in $P_0^\tau(G)=\tau P_0^1(G)$ is a sum of $\tau$ integral vectors in $P_0^1(G)$, in other words, if $SCO(G)$ forms a Hilbert basis. A finite set of integral vectors $x_1,...,x_n$ forms a \textit{Hilbert basis} if every integral vector in $\cone(\{x_1,...,x_n\})$ is a nonnegative integral combination of $x_1,...,x_n$ (see e.g.~\cite{schrijver1998theory}). This is not the case due to the following equivalence.

\begin{theorem}\label{thm:eq_E-G}
    The Edmonds-Giles Conjecture \ref{conj:E-G} is true if and only if for every undirected graph $G$ and an integer $\tau>0$, every $\tau$-SCO can be decomposed into $\tau$ SCO's.
\end{theorem}
The counterexample (in Figure~\ref{fig:strongly_connected_orientation}) to the Edmonds-Giles Conjecture \ref{conj:E-G} discovered by Schrijver  \cite{schrijver1980counterexample} can be translated to disprove the statement that every $\tau$-SCO can be decomposed into $\tau$ SCO's. Let $x\in \Z^{E^+\cup E^-}$ be defined by $x_{e}=1$ if $e$ is solid, $x_e=2$ if $e$ is dashed, and $x_e=0$ for the reverse of the dashed arcs (which we do not draw here). The vector $x$ is a $2$-SCO but it cannot be decomposed into $2$ strongly connected orientations.
\begin{figure}[htbp]
	\centering
	\includegraphics[scale=0.15]{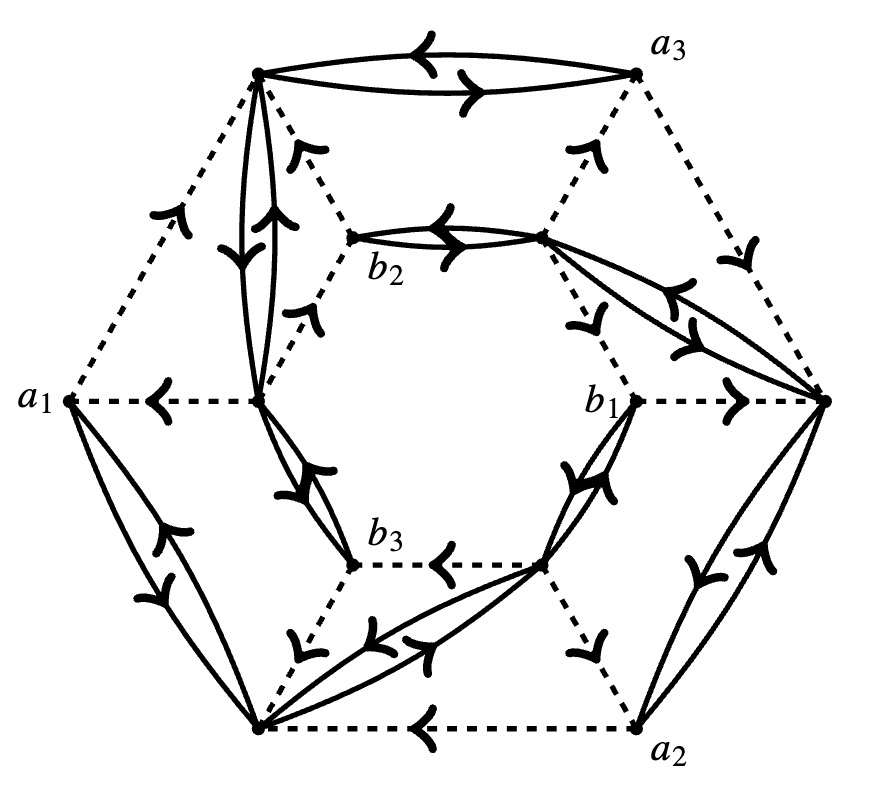}
 \caption{The solid arcs with weight $1$ and dashed arcs with weight $2$ cannot be decomposed into $2$ SCO's $O_1,O_2$. Assume for a contradiction that $O_1,O_2$ exist. The dashed arcs have their orientation fixed in both $O_i$. Three paths consisting of solid arcs between $a_i,b_i$ have to be directed paths in both $O_i$, otherwise there is a trivial dicut along the paths in some $O_i$. Both $O_i$ need to enter the inner hexagon from the outer hexagon, which means each $O_i$ should have at least one directed path oriented as $a_i\rightarrow b_i$. Thus, one $O_i$ has exactly one directed path oriented as $b_i\rightarrow a_i$ and two oriented as $a_i\rightarrow b_i$. Assume $O_1$ has orientation $b_1\rightarrow a_1, a_2\rightarrow b_2$ and $a_3\rightarrow b_3$. This leaves no arc to go from the left half to the right half of the graph, a contradiction to $O_1$ being an SCO.}
 \label{fig:strongly_connected_orientation}
\end{figure}

However, slightly revising the statement, we obtain an equivalent form of Woodall's conjecture \ref{conj:Woodall}, which is still of interest. A nowhere-zero $\tau$-strongly connected orientation (nowhere-zero $\tau$-SCO) of $G$ is an integral vector in the following polytope:
\begin{equation}
    \begin{aligned}
        P_1^\tau(G):=\big\{x\in\mathbb{R}^{E^+\cup E^-}\big|~&x_{e^+}\geq 1,~ x_{e^-}\geq 1, ~\forall e\in E,\\
        &x_{e^+}+x_{e^-}=\tau,~\forall e\in E,\\
        &x(\delta^+(U))\geq \tau,~ \forall \emptyset \neq U\subsetneqq V\big\}.
    \end{aligned}
\end{equation}
The only difference between $P_0^\tau(G)$ and $P_1^\tau(G)$ is that each entry of the integral vectors in $P_1^\tau(G)$ is nonzero. Note that $P_1^\tau(G)$ can be obtained from the following box-constrained submodular flow via a linear transformation $x\mapsto (x,\tau\cdot\mathbf{1}-x)$.
\begin{equation*}
    \begin{aligned}
        \big\{x\in\mathbb{R}^{E^+}\big|~&1\leq x_{e^+}\leq \tau-1,~\forall e\in E,\\
        &x(\delta_{E^+}^-(U))-x(\delta_{E^+}^+(U))\leq \tau(|\delta_{E^+}^-(U)|-1), \forall \emptyset \neq U\subsetneqq V\big\}.
    \end{aligned}
\end{equation*}
It follows from \Cref{thm:submodular-flow} that $P_1^\tau(G)$ is integral.
Similarly, define
    \begin{equation}\label{eq:def_P1}
    \begin{aligned}
        P_1(G):=\big\{x\in\mathbb{R}^{E^+\cup E^-}\big|~\exists \tau, ~s.t. ~&x_{e^+}\geq 1,~ x_{e^-}\geq 1, ~\forall e\in E,\\
        &x_{e^+}+x_{e^-}=\tau,~\forall e\in E,\\
        &x(\delta^+(U))\geq \tau,~ \forall \emptyset \neq U\subsetneqq V\big\}.
    \end{aligned}
    \end{equation}

We have the following reformulation of Woodall's conjecture.

\begin{theorem}\label{thm:eq_Woodall}
    Woodall's Conjecture \ref{conj:Woodall} is true if and only if for every undirected graph $G$ and an integer $\tau>0$, every nowhere-zero $\tau$-SCO can be decomposed into $\tau$ SCO's.
\end{theorem}

We will first prove Theorem \ref{thm:eq_Woodall} and modify the proof to prove Theorem \ref{thm:eq_E-G}. Our proof of Theorem \ref{thm:eq_Woodall} is inspired by Schrijver's Theorem \ref{thm:eq_strenghthening}. Schrijver's reformulation essentially covers the special case when $\bar{w}^D\in\Z^{E^+\cup E^-}$ with $\bar{w}^D_{e^+}=\tau-1,\forall e^+\in E^+$ and $\bar{w}^D_{e^-}=1,\forall e^-\in E^-$ in Theorem \ref{thm:eq_Woodall} (one can easily verify that $\bar{w}^D$ is a nowhere-zero $\tau$-SCO). Indeed, in  digraph $D=(V,A)$, let $E^+=A$, $E^-=A^{-1}$. It follows from the one-to-one mapping between strengthenings and SCO's that $A$ can be partitioned into $\tau$ strengthenings if and only if $\bar{w}$ can be decomposed into $\tau$ SCO's.  We generalize the weights to be any nowhere-zero $\tau$-SCO of $D$, and thus give a stronger consequence of Woodall's conjecture. By allowing the entries of a $\tau$-SCO to take $0$ values, we give an equivalent statement of the Edmonds-Giles conjecture in Theorem~\ref{thm:eq_E-G}, showing a contrast between the two conjectures.

\begin{proof}[Proof of Theorem \ref{thm:eq_Woodall}]
    We first prove the ``if'' direction. Let $D=(V,A)$ be a digraph (e.g. Figure~\ref{fig:sco_equivalence}-(1)) whose underlying undirected graph is $G=(V,E)$. Let $\tau$ be the size of a minimum dicut of $D$. We assume $\tau\geq 2$ w.l.o.g. and this implies that the size of minimum cut of $D$ is also greater than or equal to $2$. By making two copies of each edge of $G$ and orienting them oppositely, we obtain $\vec{G}=(V,E^+\cup E^-)$. For convenience we will assume that $E^+=A$ and $E^-=A^{-1}$. Take $x\in\Z^{E^+\cup E^-}$ such that $x_{e^+}=\tau-1$, $x_{e^-}=1$ for every $e\in E$ (as shown in Figure~\ref{fig:sco_equivalence}-(2)). We claim that $x\in P_1^\tau(G)$. The only nontrivial constraint to prove is $x(\delta_{\vec{G}}^+(U))\geq \tau$ for every $\emptyset \neq U\subsetneqq V$. If $\delta_D^-(U)$ is a dicut such that $\delta_D^+(U)=\emptyset$, then $x(\delta_{\vec{G}}^+(U))=x(\delta_{E^-}^+(U))=|\delta_D^-(U)|\geq \tau$. Otherwise, $\delta_D^+(U)\neq\emptyset$ and thus $\delta_{\vec{G}}^+(U)$ contains at least one arc in $E^+$. Moreover, since $|\delta_D(U)|\geq 2$, one has $x(\delta_{\vec{G}}^+(U))=x(\delta_{E^+}^+(U))+x(\delta_{E^-}^+(U))\geq (\tau-1)+1=\tau$. Thus, $x\in P_1^\tau(G)$. By the assumption, $x=\sum_{i=1}^{\tau}\chi_{O_i}$ where each $O_i$ is a strongly connected orientation.
    Take $J_i=\{e^+\in E^+ \mid \chi_{O_i}(e^-)=1, e^-\in E^-\}$. Note that $(A\setminus J_i)\cup (J_i^{-1})=O_i$. Therefore, $(A\setminus J_i)\cup (J_i^{-1})$ is strongly connected, which means $J_i$ is a strengthening of $D$, and thus a dijoin of $D$. Since $x_{e^-}=1$ for each $e\in E$, $J_i$'s are disjoint. Thus we get $\tau$ disjoint dijoins of $D$.
\begin{figure}[htbp]
	\centering
	\includegraphics[scale=0.23]{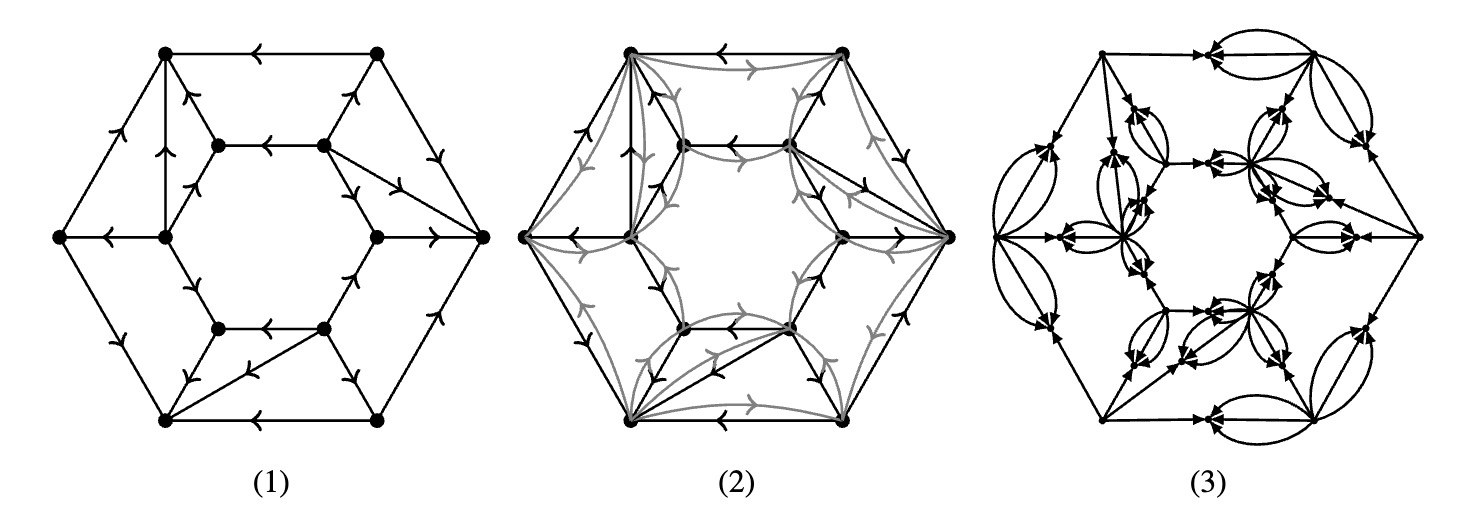}
 \caption{(1) is a digraph with minimum dicut size $4$. In (2) the weights of black arcs are $3$ and the weights of gray arcs are $1$. This figure illustrates how to convert from a digraph $D$ (1) to a weighted digraph $\vec{G}$ (2) in the first part of the proof of Theorem \ref{thm:eq_Woodall} and how to convert from a weighted digraph $\vec{G}$ (2) to a digraph $D$ (3) in the second part of the proof of Theorem \ref{thm:eq_Woodall}.}
 \label{fig:sco_equivalence}
\end{figure}

We now prove the ``only if'' direction. Given an undirected graph $G=(V,E)$, consider the corresponding directed graph $\vec{G}=(V,E^+\cup E^-)$ with each edge of $E$ copied and oppositely oriented. For an integral $x\in P_1^\tau(G)$, (e.g. Figure~\ref{fig:sco_equivalence}-(2)) construct a new digraph $D$ from $G$ in the following way. For each edge $e=\{u,v\}\in E$ where $e^+=(u,v)$ and $e^-=(v,u)$, add a node $w_e$, add $x_{e^+}\geq 1$ arcs from $u$ to $w_{e}$ and $x_{e^-}\geq 1$ arcs from $v$ to $w_{e}$, and delete $e$ (as shown in Figure~\ref{fig:sco_equivalence}-(3)).
    We claim that the size of a minimum dicut of $D$ is $\tau$. Every vertex $w_e$ induces a dicut $\delta^-_D(\{w_e\})$ of size $x_{e+}+x_{e^-}=\tau$. Thus, we only need to show that the size of every dicut of $D$ is at least $\tau$. Given $U$ such that $\delta_D^-(U)=\emptyset$, if there exists $e=\{u,v\}\in E$ such that $u,v\in U$ but $w_e\notin U$, then $|\delta_D^+(U)|\geq |\delta_D^-(\{w_e\})|\geq \tau$. Thus, we may assume w.l.o.g. that for every $e=\{u,v\}\in E$ such that $u,v\in U$, we also have $w_e\in U$. Since $\delta^-_D(U)=\emptyset$, for every $e=\{u,v\}\in E$ such that $u,v\notin U$, we also have $w_e\notin U$. Moreover, for every $e=\{u,v\}\in E$ such that $u\in U$, $v\notin U$, since there is at least an arc from $v$ to $w_e$ but $\delta_D^-(U)=\emptyset$, we infer that $w_e\notin U$. Thus, $\delta_D^+(U)=\{uw_e\mid e=\{u,v\}\in E, u\in U, v\notin U\}$. Thus, by the way we construct $D$, $|\delta_D^+(U)|\geq x(\delta_{\vec{G}}^+(U))\geq \tau$. Therefore, $D$ has minimum dicut size $\tau$. By Woodall's conjecture, there exists $\tau$ disjoint dijoins $J_1,...,J_\tau$ in $D$. In particular, each dijoin intersects dicut $\delta^-_D(\{w_e\})$ exactly once   
 since $|\delta^-_D(\{w_e\})|=\tau$. Let $O_i$ be an orientation defined by $O_i:=\{e^+\mid uw_e\in J_i\}\cup \{e^-\mid vw_e\in J_i\}$. Note that $O_i$ is indeed an orientation since exactly one of $uw_e$ and $vw_e$ is in $J_i$, for every $e\in E$. We claim that each $O_i$ is a strongly connected orientation of $G$. Assume not. Then there exists $U\subseteq V$, such that $\delta_{\vec{G}}^+(U)\cap O_i=\emptyset$. Let $U':=U\cup \{w_e\mid e=\{u,v\}\in E, u,v\in U\}$. It is easy to see that $U'$ is a dicut of $D$ such that $\delta_D^-(U')=\emptyset$. It follows from $\delta_{\vec{G}}^+(U)\cap O_i=\emptyset$ that $\delta_D^+(U')\cap J_i=\emptyset$, a contradiction to $J_i$ being a dijoin of $D$. Moreover, by the way we construct $D$, for each $e^+=(u,v)$, $\sum_{i=1}^{\tau}\chi_{O_i}(e^+)=|\{J_i\mid uw_e\in J_i\}|=x_{e^+}$. For each $e^-=(v,u)$, $\sum_{i=1}^{\tau}\chi_{O_i}(e^-)=|\{J_i\mid vw_e\in J_i\}|=x_{e^-}$. Therefore, $\sum_{i=1}^{\tau}\chi_{O_i}=x$. This ends the proof of this direction.
\end{proof}

It is implicit in the proof of Theorem \ref{thm:eq_Woodall} that the equivalence also holds true when restricting to a specific $\tau$:  Woodall's conjecture is true for every digraph with minimum size of a dicut $\tau$ if and only if for every undirected graph $G$, every vector $x\in P_1^\tau(G)$ can be decomposed into $\tau$ strongly connected orientations. A direct consequence of Theorem \ref{thm:eq_Woodall} is the following, which has been noted by A. Frank (see e.g. \cite{schrijver2003combinatorial} Theorem 56.3).
\begin{corollary}\label{cor:tau=2}
    Woodall's conjecture is true for $\tau=2$.
\end{corollary}
\begin{proof}
    As we remarked earlier about Theorem \ref{thm:eq_Woodall} for a specific $\tau$, Woodall's conjecture is true for every digraph with minimum dicut size $\tau=2$ if and only if for every undirected graph $G$, every vector $x\in P_1^2$ can be decomposed into $2$ strongly connected orientations. However, the all one vector $\mathbf{1}$ is the only possible vector in $P_1^2$. Clearly, $\mathbf{1}\in P_1^2$ if and only if $G$ is $2$-edge-connected. By the classical result of Robbins~\cite{robbins1939theorem}, every $2$-edge-connected graph has a strongly connected orientation $O$. Therefore, $\mathbf{1}=\chi_O+\chi_{O^{-1}}$ is a valid decomposition.  
\end{proof}

 To prove Theorem \ref{thm:eq_E-G}, we need a structural lemma.    
    \begin{lemma}\label{lemma:near_dicut}
        Let $D=(V,A)$ be a digraph with weight $w\in\{0,1\}^A$ and assume that the minimum weight of a dicut is $\tau\geq 2$. Let $e\in A$ be an arc such that $w_e=1$. If there exists a cut $\delta_D(U)$ such that $\delta_D^+(U)=\{e\}$ and $w(\delta_D^-(U))=0$, then $e$ is not contained in any minimum dicut of $D$.
    \end{lemma}
    \begin{proof}
        Suppose not. Then there exists a dicut $\delta_D^-(W)$ such that $w(\delta_D^-(W))=\tau$ and $e\in \delta_D^-(W)$. Let $D'$ be obtained from $D$ by deleting $e$. Then $\delta_{D'}^-(U)$ becomes a dicut of $D'$. Therefore, $\delta_{D'}^-(U\cap W)$ and $\delta_{D'}^-(U\cup W)$ are both dicuts of $D'$. However, since $e$ leaves $U$ and enters $W$, $e$ goes from $U\setminus W$ to $W\setminus U$. Thus, $e\notin \delta_D(U\cap W)$ and $e\notin \delta_D(U\cup W)$. Therefore, both $\delta_{D}^-(U\cap W)$ and $\delta_{D}^-(U\cup W)$ are dicuts of $D$. Moreover, $w(\delta_{D}^-(U\cap W))+w(\delta_{D}^-(U\cup W))=w(\delta_{D}^-(U))+w(\delta_{D}^-(W))-1=\tau-1$.
        It follows that $w(\delta_{D}^-(U\cap W))\leq \tau-1$ and $w(\delta_{D}^-(U\cup W))\leq \tau-1$. Notice that either $U\cap W\neq \emptyset$ or $U\cup W\neq V$: otherwise, $e$ is a bridge of $D$, contradicting $\tau\geq 2$. Therefore, either $\delta_{D}^-(U\cap W)$ or $\delta_{D}^-(U\cup W)$ violates the assumption that the size of a minimum dicut is $\tau$, a contradiction.
    \end{proof}

\begin{proof}[Proof of Theorem \ref{thm:eq_E-G}]
    We modify the proof of Theorem \ref{thm:eq_Woodall} to prove Theorem \ref{thm:eq_E-G}. We first prove the ``if'' direction.
    Let $D=(V,A)$ be a digraph with weight $w\in\{0,1\}^A$ and minimum dicut $\tau\geq 2$. We can assume there is no arc $e\in A$ with weight $1$ such that there exists a cut $\delta_D(U)$ such that $\delta_D^+(U)=\{e\}$ and $w(\delta_D^-(U))=0$. For otherwise, by Lemma \ref{lemma:near_dicut}, $e$ is not contained in any minimum dicut, which means we can set the weight of $e$ to be $0$ without decreasing the size of a minimum dicut. Every packing of $\tau$ dijoins in the new graph will be a valid packing of $\tau$ dijoins in $D$ with weight $w$. 
    
    Let $G$ be the underlying undirected graph of $D$ and $\vec{G}=(V,E^+\cup E^-)$ be defined as before such that $E^+=A$ and $E^-=A^{-1}$. Define $x\in\Z^{E^+\cup E^-}$ as follows. For weight $1$ arcs $e^+\in A$, we define $x_{e^+}=\tau-1$ and $x_{e^-}=1$ as before. For the weight $0$ arcs $e^+\in A$, we define $x_{e^+}=\tau$ and $x_{e^-}=0$. To argue that $x(\delta^+_{\vec{G}}(U))\geq \tau$ for every $\emptyset \neq U\subsetneqq V$, if $\delta_D^-(U)$ is a dicut we follow the same approach as in the proof of Theorem \ref{thm:eq_Woodall}. Therefore, without loss of generality, we assume there exists at least one arc $e^+\in E^+$ in $\delta_{D}^+(U)$. If there exists such an arc with $w(e^+)=0$, then $x(\delta_{\vec{G}}^+(U))\geq x_{e^+}\geq \tau$. Otherwise, all the arcs in $\delta_D^+(U)$ have weight $1$. If there exist at least $2$ arcs of weight $1$ in $\delta_{D}(U)$, we follow the same argument as in the earlier proof. The only case left is when $\delta_{D}^+(U)$ is a single arc of weight $1$ and all the arc in $\delta_D^-(U)$ has weight $0$, which has been excluded in the beginning. Therefore, we have proved $x(\delta^+_{\vec{G}}(U))\geq \tau$ for every $\emptyset \neq U\subsetneqq V$, which implies $x\in P_0^\tau(G)$. By the assumption, $x=\sum_{i=1}^{\tau}\chi_{O_i}$ where each $O_i$ is a strongly connected orientation. We define the dijoins in the same way as the other proof. Note that the dijoins are disjoint and never use weight $0$ arcs. Therefore, we find $\tau$ dijoins that form a valid packing of graph $D$ with weight $w$.

    Next, we prove the ``only if'' direction. Given an undirected graph $G=(V,E)$, the corresponding $\vec{G}=(V,E^+\cup E^-)$, and an integral $x\in P_0^\tau(G)$, we construct weighted digraph $D$ as follows. For an edge $e=\{u,v\}$ such that $x_{e^+}, x_{e^-}\geq 1$, we construct node $w_e$ and arcs $uw_e$, $vw_e$ in the same way as in the proof of Theorem \ref{thm:eq_Woodall}. For an arc $e^+=(u,v)$ such that $x_{e^+}=\tau$, $x_{e^-}=0$, we add node $w_e$, add $\tau$ arcs of weight $1$ from $u$ to $w_e$ and add a weight $0$ arc from $v$ to $w_e$. Similarly, for $e^+=(u,v)$ with $x_{e^+}=0$, $x_{e^-}=\tau$, we add node $w_e$, add a weight $0$ arc from $u$ to $w_e$ and $\tau$ arcs of weight $1$ from $v$ to $w_e$. The same argument applies to see that the minimum size of a dicut in $D$ is $\tau$. By the Edmonds-Giles conjecture, we can find $\tau$ disjoint dijoins in the weighted digraph $D$. As before, we can find $\tau$ strongly connected orientations accordingly that sum up to $x$.
\end{proof}


    


\section{Strongly Connected Digraphs and Decomposition}\label{sec:Decompose_SCD}
In this section, we will revisit strongly connected digraphs and study their polytopes. We will see the difference between the polytopes of strongly connected orientations and strongly connected digraphs. We also give a decomposition result of weighted strongly connected digraphs whose weights are ``nowhere-zero". 
\subsection{Strongly Connected Digraph Polytopes}
Given an undirected graph $G=(V,E)$ and $\vec{G}=(V,E^+\cup E^-)$ obtained by copying and oppositely orienting each edge of $G$, let $SCD(\vec{G})$ be the polytope of strongly connected (sub)digraphs of $G$. Let $\chi_F$ be the characteristic vector of a strongly connected subgraph $F$ of $\vec{G}$.
     \begin{equation}
    \begin{aligned}
        SCD(\vec{G}):=\big\{x\in\{0,1\}^{E^+\cup E^-}\big|&x=\chi_F \text{ for some }  \text{strongly connected subgraph $F$ of $G$}\big\}.
    \end{aligned}
    \end{equation}
It coincides with the $0,1$ vectors in the following polyhedron:
\[
\begin{aligned}
    \big\{x\in\mathbb{R}^{E^+\cup E^-}\big|~&x_{e^+}\geq 0,~ x_{e^-}\geq 0, ~\forall e\in E\\
        &x(\delta^+(U))\geq 1, \forall \emptyset \neq U\subsetneqq V\big\}.
\end{aligned}
\]

More generally, for a positive integer $\tau$, a $\tau$-strongly-connected (sub)digraph ($\tau$-SCD) of $G$ is an integral vector in
    \begin{equation}
    \begin{aligned}
        Q_0^\tau(G):=\big\{x\in\mathbb{R}^{E^+\cup E^-}\big|~&x_{e^+}\geq 0,~ x_{e^-}\geq 0, ~\forall e\in E\\
        &x(\delta^+(U))\geq \tau,~ \forall \emptyset \neq U\subsetneqq V\big\}
    \end{aligned}
    \end{equation}
In contrast to $P_0^1(G)$ (defined in \eqref{eq:def_P0^tau}) the natural LP relaxation for $SCO(G)$ which is integral, $Q_0^1(G)$ is not integral (e.g. $Q_0^1(G)$ for the complete graph $K_4$ has non-integral extreme points such as the one in Figure \ref{fig:scd_half_integral}.).
\begin{figure}[htbp]
	\centering
	\includegraphics[scale=0.4]{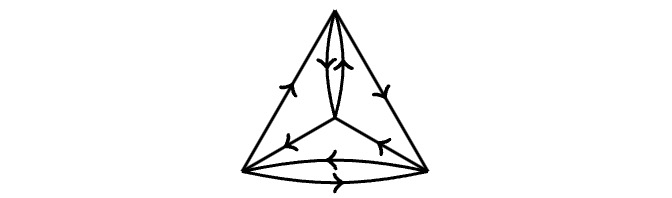}
 \caption{The value $\frac{1}{2}$ on every arc drawn and $0$ everywhere else is a half-integral extreme point of $Q_0^1(K_4)$}.
 \label{fig:scd_half_integral}
\end{figure}

    A nowhere-zero $\tau$-strongly-connected (sub)digraph (nowhere-zero $\tau$-SCD) of $G$ is an integral vector in
    \begin{equation}\label{eq:def_Q_1^k}
    \begin{aligned}
        Q_1^\tau(G):=\big\{x\in\mathbb{R}^{E^+\cup E^-}\big|~&x_{e^+}\geq 1,~ x_{e^-}\geq 1, ~\forall e\in E\\
        &x(\delta^+(U))\geq \tau,~ \forall \emptyset \neq U\subsetneqq V\big\}.
    \end{aligned}
    \end{equation}
Note that $Q_0^\tau(G)$ (resp. $Q_1^\tau(G)$) can be obtained from $P_0^\tau(G)$ (resp. $P_1^\tau(G)$) by relaxing the constraints $x_{e^+}+x_{e^-}=\tau, \forall e\in E$. In particular, in a strongly connected digraph, for each edge $e\in E$, we can use both $e^+$ and $e^-$ or neither of them, while in a strongly connected orientation we need to use exactly one of them. 


\begin{figure}[htbp]
	\centering
	\includegraphics[scale=0.2]{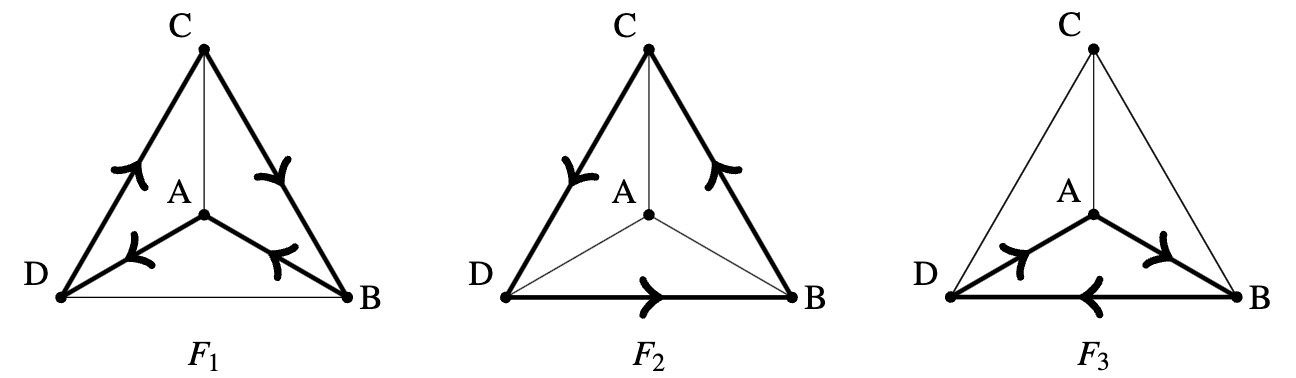}
 \caption{$\mathbf{1}\in Q_1^3(K_4)$ cannot be decomposed into $3$ strongly connected digraphs.}
 \label{fig:scd}
\end{figure}

\paragraph{Remark.} In general, it is not true that an arbitrary integral vector in $Q_1^\tau(G)$ can be decomposed into $\tau$ disjoint strongly connected digraphs. A counterexample is given in Figure \ref{fig:scd}. The underlying undirected graph $G$ is the complete graph $K_4$. Let the weight $w$ be the all one vector $\mathbf{1}$. It satisfies $w(\delta_{\vec{G}}(U))\geq 3, \forall \emptyset \neq U\subsetneqq V$ and thus $w\in Q_1^3(K_4)$. Assume that $w$ can be decomposed into $3$ strongly connected digraphs $F_1,F_2,F_3$. There are $3$ arc-disjoint directed paths from $D$ to $B$: $D\rightarrow B$, $D\rightarrow A\rightarrow B$ and $D\rightarrow C\rightarrow B$, so each strongly connected digraph should use exactly one of them. So should the $3$ arc-disjoint directed paths from $B$ to $D$. If $D\rightarrow B$ and $B\rightarrow D$ are used in the same strongly connected digraph $F_1$, then $F_2$ and $F_3$ will use all the arcs connecting $\{A,C\}$ to $\{B,D\}$ which leaves $\{A,C\}$ disconnected to $\{B,D\}$ in $F_1$, a contradiction. Thus, exactly one of them, say $F_1$, uses length $2$ paths in both directions, and the other two, $F_2,F_3$, use length $1$ path in one direction and length $2$ path in the other (see, e.g. Figure \ref{fig:scd}). Then there will be a directed cycle of length $3$ in each of $F_2$ and $F_3$. The only arcs left are $A\rightarrow C$ and $C\rightarrow A$. Both $F_2$ and $F_3$ need both arcs to strongly connect the remaining isolated node to the directed cycle. This is impossible.

\subsection{Decomposing Nowhere-Zero $\tau$-SCD's}
The proof of Theorem \ref{thm:decompose_SCD_weightD} relies heavily on the special structure of the weight $w^D$ associated with $D$. It does not generalize to decomposing an arbitrary $\tau$-SCD into two $\rounddown{\frac{\tau}{k}}$-SCD's. In fact, it is even open whether there exists a sufficiently large integer $\tau$ such that every $\tau$-SCD can be decomposed into two disjoint SCD's \cite{bang2004decomposing}.   However, we provide a counterpart of Theorem~\ref{thm:decompose_SCD_weightD} about decomposing an arbitrary nowhere-zero $\tau$-SCD's.

\begin{theorem}\label{thm:decompose_SCD}
    For an undirected graph $G=(V,E)$, if it admits a nowhere-zero circular $k$-flow for some rational number $k\geq 2$, then every nowhere-zero $\tau$-SCD of $\vec{G}$ can pack two $\rounddown{\frac{\tau}{k+1}}$-SCD's.
\end{theorem}
    \begin{proof}
        Let $E^+$ and $f:E^+\rightarrow [1,k-1]$ be a nowhere-zero circular $k$-flow of $G$. Let $E^-$ be obtained by reversing the arcs of $E^+$. Let $\vec{G}=(V,E^+\cup E^-)$. By Lemma \ref{lem:key}, both $E^+$ and $E^-$ are $k$-cut-balanced orientations. Let $w\in Q_1^\tau(G)$ be an integral nowhere-zero $\tau$-SCD. We construct $x\in \Z^{E^+\cup E^-}$ in the following way:
        \begin{equation*}
            \begin{aligned}
                x_e=\begin{cases}
                    \roundup{\frac{w_e}{2}}, & e\in E^+\\
                    \rounddown{\frac{w_e}{2}}, & e\in E^-
                \end{cases}
            \end{aligned},
            \quad\text{and equivalently}\quad 
            \begin{aligned}
                (w-x)_e=\begin{cases}
                    \rounddown{\frac{w_e}{2}}, & e\in E^+\\
                    \roundup{\frac{w_e}{2}}, & e\in E^-
                \end{cases}.
            \end{aligned}
        \end{equation*}
    We claim that $x$ and $(w-x)$ are both $\rounddown{\frac{\tau}{k+1}}$-SCD's. Let $\emptyset \neq U \subsetneqq V$. We discuss two cases. 
        
        Suppose $|\delta_G(U)|\geq \frac{k}{k+1}\tau$. Then, the weights $x_e$ of all arcs $e\in\delta_{E^+}^+(U)$ will be rounded up to $\roundup{\frac{w_e}{2}}\geq \roundup{\frac{1}{2}}=1$. Therefore, $x(\delta_{\vec{G}}^+(U))\geq x(\delta_{E^+}^+(U)) \geq |\delta_{E^+}^+(U)|\geq \frac{1}{k}|\delta_{G}(U)|\geq \frac{\tau}{k+1}$.
        Symmetrically, the weights $(w-x)_e$ of all arcs $e\in\delta_{E^-}^+(U)$ will be rounded up to $\roundup{\frac{w_e}{2}}\geq \roundup{\frac{1}{2}}=1$. Therefore, $(w-x)(\delta_{\vec{G}}^+(U))\geq (w-x)(\delta_{E^-}^+(U))\geq  |\delta_{E^-}^+(U)|\geq \frac{1}{k}|\delta_{G}(U)|\geq \frac{\tau}{k+1}$. 
        
        Suppose $|\delta_G(U)|<\frac{k}{k+1}\tau$. Then, \[
        \begin{aligned}
            x(\delta_{\vec{G}}^+(U))=&\roundup{\frac{w}{2}}(\delta_{E^+}^+(U))+\rounddown{\frac{w}{2}}(\delta_{E^-}^+(U))\\
            \geq& \frac{w}{2}(\delta_{E^+}^+(U))+\big(\frac{w-1}{2}\big)(\delta_{E^-}^+(U))\\
            =&\frac{1}{2}w(\delta_{\vec{G}}^+(U))-\frac{1}{2}|\delta_{E^-}^+(U)|\\
        \geq& \frac{1}{2}w(\delta_{\vec{G}}^+(U))-\frac{1}{2}\cdot\frac{k-1}{k}|\delta_G(U)|\\
        \geq& \frac{1}{2}\tau-\frac{1}{2}\cdot\frac{k-1}{k}\cdot\frac{k}{k+1}\tau\\
        =&\frac{\tau}{k+1},
        \end{aligned}
        \]
        where the second inequality follows from Lemma \ref{lem:key} and the third inequality uses the fact that $w\in Q_1^\tau(G)$. This shows that $x$ is a $\rounddown{\frac{\tau}{k+1}}$-SCD.
        By interchanging $E^+$ and $E^-$, a symmetrical argument applies to prove that $(w-x)$ is also a $\rounddown{\frac{\tau}{k+1}}$-SCD.
\end{proof}


The above theorem along with the technique of pairing in and out $r$-arborescences gives the following result.

\begin{theorem}\label{thm:pack_SCD1}
    For an undirected graph $G=(V,E)$, if it admits a nowhere-zero circular $k$-flow for some rational number $k\geq 2$, then every nowhere-zero $\tau$-SCD of $\vec{G}$ can pack $\rounddown{\frac{\tau}{k+1}}$ SCD's.
\end{theorem}

\section{Conclusions and Discussions}
We showed that every digraph with minimum dicut size $\tau$ can pack $\rounddown{\frac{\tau}{6}}$ dijoins, or $\rounddown{\frac{\tau p}{2p+1}}$ dijoins when the digraph is $6p$-edge-connected. 
The existence of nowhere-zero circular $k$-flow for a smaller $k$ ($<6$) when special structures are imposed on the underlying undirected graphs would lead to a better ratio, i.e., $\rounddown{\frac{\tau}{k}}$, approximate packing of dijoins for those digraphs. 
The limitation of this approach is that we cannot hope that nowhere-zero $2$-flows always exist because this is equivalent to the graph being Eulerian. Thus, bringing the number up to $\rounddown{\frac{\tau}{2}}$ disjoint dijoins would be challenging using this approach. However, it is necessary for Woodall's conjecture to be true that every digraph with minimum dicut size $\tau$ contains two disjoint $\rounddown{\frac{\tau}{2}}$-dijoins. Therefore, new ideas are needed to prove or disprove whether such a decomposition exists.

The careful reader may have noticed that the approach only works for the unweighted case. However, by a slight modification of the argument, it extends to the weighted case when the underlying undirected graph of the weight-$1$ arcs is bridgeless. In this case, we can find a nowhere-zero $k$-flow on the weight-$1$ arcs and construct the decomposition of weight-$1$ arcs the same way as in Theorem \ref{thm:decompose_SCD_weightD}. However, unfortunately, in general, the underlying graph of weight-$1$ arcs may have bridges, in which case the above argument does not work. 
A proper analogue of nowhere-zero flows in mixed graphs should be developed, which motivates us to raise the following open problems.
    
Given a mixed graph $M=(V,E\cup A)$, where $A$ is directed and $E$ is undirected, a \textit{pseudo dicut} is an edge subset of the form $\delta_E(U)$ such that $\delta^-_A(U)=\emptyset$. Let $\tau$ be the minimum size of a pseudo dicut.
\begin{conjecture}\label{con:weighted_strong}
There is a constant $k>2$ such that for every mixed graph $M=(V,E\cup A)$, there always exists an orientation $E^+$ of the edges $E$ such that 
$|\delta_{E^+}^+(U)|,|\delta_{E^+}^-(U)|\geq \rounddown{\frac{1}{k}|\delta_E(U)|}$ for every pseudo dicut $\delta_E(U)$.
\end{conjecture}
Note that the counterexample (in Figure~\ref{fig:strongly_connected_orientation}) to Edmonds-Giles conjecture implies $k> 2$. A weaker question is the following.
\begin{conjecture}\label{con:weighted_weak}
    There is a constant $k>2$ such that for every mixed graph $M=(V,E\cup A)$ with minimum pseudo dicut size at least $k$, there always exists an orientation $E^+$ such that $|\delta_{E^+}^+(U)|,|\delta_{E^+}^-(U)|\geq 1$ for every pseudo dicut $\delta_E(U)$.
\end{conjecture}
Again, the counterexample to Edmonds-Giles conjecture (in Figure~\ref{fig:strongly_connected_orientation}) implies $k>2$. An affirmative answer to Conjecture \ref{con:weighted_strong} would immediately lead to a packing of $\rounddown{\frac{\tau}{k}}$ dijoins in the weighted case. An affirmative answer to Conjecture \ref{con:weighted_weak} would immediately lead to a packing of $2$ dijoins in the weighted case when $\tau\geq k$.

\section*{Acknowledgements}
The first author is supported by the US Office of Naval Research under award number N00014-22-1-2528, the second author is supported by a Balas Ph.D. fellowship from the Tepper School, Carnegie Mellon University, and the third author is supported by the US Office of Naval Research under award number N00014-21-1-2243 and the Air Force Office of Scientific Research under award number FA9550-23-1-0031. We would like to thank Ahmad Abdi, Krist\'of B\'erczi, Karthik Chandrasekaran, Bruce Shepherd, Olha Silina and Michael Zlatin for discussions of various aspects of Woodall's conjecture and their useful feedback on this paper. We also thank two anonymous reviewers whose feedback helped improve the presentation of this paper.

\newpage
\bibliographystyle{plain} 
\bibliography{reference}

\begin{thebibliography}{10}

\bibitem{abdi2024arc}
Ahmad Abdi, G{\'e}rard Cornu{\'e}jols, and Giacomo Zambelli.
\newblock Arc connectivity and submodular flows in digraphs.
\newblock {\em Combinatorica}, pages 1--22, 2024.

\bibitem{abdi2023packing}
Ahmad Abdi, G{\'e}rard Cornu{\'e}jols, and Michael Zlatin.
\newblock On packing dijoins in digraphs and weighted digraphs.
\newblock {\em SIAM Journal on Discrete Mathematics}, 37(4):2417--2461, 2023.

\bibitem{bang2009disjoint}
J{\o}rgen Bang-Jensen and Matthias Kriesell.
\newblock Disjoint sub (di) graphs in digraphs.
\newblock {\em Electronic Notes in Discrete Mathematics}, 34:179--183, 2009.

\bibitem{bang2004decomposing}
J{\o}rgen Bang-Jensen and Anders Yeo.
\newblock Decomposing k-arc-strong tournaments into strong spanning
  subdigraphs.
\newblock {\em Combinatorica}, 24:331--349, 2004.

\bibitem{baum1978integer}
Stephen Baum and Leslie~E Trotter.
\newblock Integer rounding and polyhedral decomposition for totally unimodular
  systems.
\newblock In {\em Optimization and Operations Research: Proceedings of a
  Workshop Held at the University of Bonn, October 2--8, 1977}, pages 15--23.
  Springer, 1978.

\bibitem{chudnovsky2016disjoint}
Maria Chudnovsky, Katherine Edwards, Ringi Kim, Alex Scott, and Paul Seymour.
\newblock Disjoint dijoins.
\newblock {\em Journal of Combinatorial Theory, Series B}, 120:18--35, 2016.

\bibitem{CLR}
G\'erard Cornu\'ejols, Siyue Liu, and R.~Ravi.
\newblock Approximately packing dijoins via nowhere-zero flows.
\newblock In {\em IPCO 2024, LNCS 14679}, pages 71--84, 2024.

\bibitem{openproblemtau3}
Matthew Devos.
\newblock Woodall’s conjecture.
\newblock \url{http://www.openproblemgarden.org/op/woodalls_conjecture}.

\bibitem{edmonds1973edge}
Jack Edmonds.
\newblock Edge-disjoint branchings.
\newblock {\em Combinatorial Algorithms}, pages 91--96, 1973.

\bibitem{edmonds1977min}
Jack Edmonds and Rick Giles.
\newblock A min-max relation for submodular functions on graphs.
\newblock In {\em Annals of Discrete Mathematics}, volume~1, pages 185--204.
  Elsevier, 1977.

\bibitem{eswaran1976augmentation}
Kapali~P Eswaran and R~Endre Tarjan.
\newblock Augmentation problems.
\newblock {\em SIAM Journal on Computing}, 5(4):653--665, 1976.

\bibitem{goddyn2001open}
Luis~A Goddyn.
\newblock Some open problems {I} like.
\newblock \url{https://www.sfu.ca/~goddyn/Problems/problems.html}.

\bibitem{goddyn1998k}
Luis~A Goddyn, Michael Tarsi, and Cun-Quan Zhang.
\newblock On (k, d)-colorings and fractional nowhere-zero flows.
\newblock {\em Journal of Graph Theory}, 28(3):155--161, 1998.

\bibitem{hoffman2003some}
Alan~J Hoffman.
\newblock Some recent applications of the theory of linear inequalities to
  extremal combinatorial analysis.
\newblock In {\em Selected Papers Of Alan J Hoffman: With Commentary}, pages
  244--258. World Scientific, 2003.

\bibitem{jaeger1976nowhere}
Fran{\c{c}}ois Jaeger.
\newblock On nowhere-zero flows in multigraphs.
\newblock {\em Proceedings, Fifth British Combinatorial Conference, Aberdeen},
  pages 373--378, 1975.

\bibitem{jaeger1976balanced}
Fran{\c{c}}ois Jaeger.
\newblock Balanced valuations and flows in multigraphs.
\newblock {\em Proceedings of the American Mathematical Society},
  55(1):237--242, 1976.

\bibitem{jaeger1979flows}
Fran{\c{c}}ois Jaeger.
\newblock Flows and generalized coloring theorems in graphs.
\newblock {\em Journal of Combinatorial Theory, series B}, 26(2):205--216,
  1979.

\bibitem{jaeger1984circular}
Fran{\c{c}}ois Jaeger.
\newblock On circular flows in graphs.
\newblock In {\em Finite and Infinite Sets}, pages 391--402. Elsevier, 1984.

\bibitem{jaeger1988nowhere}
Fran{\c{c}}ois Jaeger.
\newblock Nowhere-zero flow problems.
\newblock {\em Selected Topics in Graph Theory}, 3:71--95, 1988.

\bibitem{kiraly2007result}
Tam{\'a}s Kir{\'a}ly.
\newblock A result on crossing families of odd sets.
\newblock 2007.
\newblock \url{https://egres.elte.hu/tr/egres-07-10.pdf}.

\bibitem{lovasz2013nowhere}
L{\'a}szl{\'o}~Mikl{\'o}s Lov{\'a}sz, Carsten Thomassen, Yezhou Wu, and
  Cun-Quan Zhang.
\newblock Nowhere-zero 3-flows and modulo k-orientations.
\newblock {\em Journal of Combinatorial Theory, Series B}, 103(5):587--598,
  2013.

\bibitem{meszaros2015note}
Andr{\'a}s M{\'e}sz{\'a}ros.
\newblock A note on disjoint dijoins.
\newblock {\em Combinatorica}, 38(6):1485--1488, 2018.

\bibitem{nash1964decomposition}
C~St~JA Nash-Williams.
\newblock Decomposition of finite graphs into forests.
\newblock {\em Journal of the London Mathematical Society}, 1(1):12--12, 1964.

\bibitem{robbins1939theorem}
Herbert~Ellis Robbins.
\newblock A theorem on graphs, with an application to a problem of traffic
  control.
\newblock {\em The American Mathematical Monthly}, 46(5):281--283, 1939.

\bibitem{schrijverobervation}
Alexander Schrijver.
\newblock Observations on {Woodall’s} conjecture.
\newblock \url{https://homepages.cwi.nl/~lex/files/woodall.pdf}.

\bibitem{schrijver1980counterexample}
Alexander Schrijver.
\newblock A counterexample to a conjecture of {Edmonds and Giles}.
\newblock {\em Discrete Mathematics}, 32:213--214, 1980.

\bibitem{schrijver1998theory}
Alexander Schrijver.
\newblock {\em Theory of linear and integer programming}.
\newblock John Wiley \& Sons, 1998.

\bibitem{schrijver2003combinatorial}
Alexander Schrijver et~al.
\newblock {\em Combinatorial optimization: polyhedra and efficiency},
  volume~24.
\newblock Springer, 2003.

\bibitem{seymour1981nowhere}
Paul~D Seymour.
\newblock Nowhere-zero 6-flows.
\newblock {\em Journal of Combinatorial Theory, series B}, 30(2):130--135,
  1981.

\bibitem{shepherd2005visualizing}
Bruce Shepherd and Adrian Vetta.
\newblock Visualizing, finding and packing dijoins.
\newblock {\em Graph Theory and Combinatorial Optimization}, pages 219--254,
  2005.

\bibitem{thomassen2012weak}
Carsten Thomassen.
\newblock The weak 3-flow conjecture and the weak circular flow conjecture.
\newblock {\em Journal of Combinatorial Theory, Series B}, 102(2):521--529,
  2012.

\bibitem{tutte1954contribution}
William~Thomas Tutte.
\newblock A contribution to the theory of chromatic polynomials.
\newblock {\em Canadian Journal of Mathematics}, 6:80--91, 1954.

\bibitem{woodall1978menger}
Douglas~R Woodall.
\newblock Menger and {K{\"o}nig} systems.
\newblock {\em Theory and Applications of Graphs: Proceedings, Michigan May
  11--15, 1976}, pages 620--635, 1978.

\bibitem{younger1983integer}
D.H. Younger.
\newblock Integer flows.
\newblock {\em Journal of Graph Theory}, 7(3):349--357, 1983.

\end{thebibliography}





\end{document}